\documentclass[a4paper,10pt,reqno]{amsart}
 
\usepackage[english]{babel}
\usepackage{dsfont} 
\usepackage{graphicx}
\usepackage{color}
\usepackage{hyperref}
\usepackage{cleveref}
\usepackage{enumerate}
\usepackage{amssymb,empheq} 
\usepackage{amsthm} 
\usepackage{xcolor}
\usepackage[foot]{amsaddr}

\allowdisplaybreaks

\newtheorem{thm}{Theorem}[section]
\newtheorem{thmDaniell}[thm]{Daniell's Representation Theorem}
\newtheorem{thmHahnM}[thm]{Hahn--Mazurkiewicz' Theorem}
\newtheorem{thmJankoff}[thm]{Jankoff's Theorem}
\newtheorem{thmRichter}[thm]{Richter's Theorem}
\newtheorem{thmHaus}[thm]{Hausdorff Moment Problem}
\newtheorem{lem}[thm]{Lemma}
\newtheorem{prop}[thm]{Proposition}
\newtheorem{cor}[thm]{Corollary}

\theoremstyle{definition}

\newtheorem{dfn}[thm]{Definition}

\newtheorem{exm}[thm]{Example}

\newtheorem{open}[thm]{Open Problem}

\theoremstyle{remark}
\newtheorem{rem}[thm]{Remark}

\newcommand{\exmsym}{\hfill$\circ$}

\newcommand{\cset}{\mathds{C}}
\newcommand{\kset}{\mathds{K}}
\newcommand{\nset}{\mathds{N}}
\newcommand{\pset}{\mathds{P}}
\newcommand{\qset}{\mathds{Q}}
\newcommand{\rset}{\mathds{R}}

\newcommand{\diff}{\mathrm{d}}

\newcommand{\supp}{\mathrm{supp}\,}

\newcommand{\id}{\mathrm{id}}

\newcommand{\modzero}{\mathrm{mod}0}

\newcommand{\cA}{\mathcal{A}}
\newcommand{\cB}{\mathcal{B}}

\newcommand{\cF}{\mathcal{F}}
\newcommand{\cH}{\mathcal{H}}

\newcommand{\cL}{\mathcal{L}}

\newcommand{\cP}{\mathcal{P}}

\newcommand{\cU}{\mathcal{U}}
\newcommand{\cV}{\mathcal{V}}
\newcommand{\cW}{\mathcal{W}}
\newcommand{\cX}{\mathcal{X}}
\newcommand{\cY}{\mathcal{Y}}
\newcommand{\cZ}{\mathcal{Z}}

\newcommand{\fB}{\mathfrak{B}}

\newcommand{\trans}{\leadsto}
\newcommand{\ctrans}{\overset{\text{\textit{c}}}{\leadsto}}
\newcommand{\strans}{\overset{\text{\textit{s}}}{\leadsto}}
\newcommand{\sctrans}{\overset{\text{\textit{sc}}}{\leadsto}}
\newcommand{\downtrans}{\rotatebox[origin=c]{-90}{$\trans$}}
\newcommand{\uptrans}{\rotatebox[origin=c]{90}{$\trans$}}

\author{Philipp J.\ di~Dio}

\address{Technische Universit\"at Berlin, Institut f\"ur Mathematik, Stra\ss{}e des 17.\ Juni 136, D-10623 Berlin, Germany}
\email{didio@tu-berlin.de}



\title[Transformations of Moment Functionals]{Transformations of Moment Functionals}

\begin{document}

\begin{abstract}
In measure theory several results are known how measure spaces are transformed into each other. But since moment functionals are represented by a measure we investigate in this study the effects and implications of these measure transformations to moment funcationals. We gain characterizations of moments functionals. Among other things we show that for a compact and path connected set $K\subset\rset^n$ there exists a measurable function $g:K\to [0,1]$ such that any linear functional $L:\rset[x_1,\dots,x_n]\to\rset$ is a $K$-moment functional if and only if it has a continuous extension to some $\overline{L}:\rset[x_1,\dots,x_n]+\rset[g]\to\rset$ such that \mbox{$\tilde{L}:\rset[t]\to\rset$} defined by $\tilde{L}(t^d) := \overline{L}(g^d)$ for all $d\in\nset_0$ is a $[0,1]$-moment functional (Hausdorff moment problem). Additionally, there exists a continuous function $f:[0,1]\to K$ independent on $L$ such that the representing measure $\tilde{\mu}$ of $\tilde{L}$ provides the representing measure $\tilde{\mu}\circ f^{-1}$ of $L$. We also show that every moment functional $L:\cV\to\rset$ is represented by $\lambda\circ f^{-1}$ for some measurable function $f:[0,1]\to\rset^n$ where $\lambda$ is the Lebesgue on $[0,1]$.
\end{abstract}

\maketitle

\tableofcontents

\section{Introduction}

Linear functionals $L:\cV\to\kset$ with $\kset = \rset$ or $\cset$ belong to the most important structures in mathematics, e.g.\ for separation arguments. If $\cV$ is a vector space of functions $v:\cX\to\kset$ then $L$ is called a moment functional if it is represented by a (non-negative) measure $\mu$ on $\cX$:
\[L(v)=\int_\cX v(x)~\diff\mu(x) \qquad\text{for all}\ v\in\cV.\]
If $\supp\mu\subseteq K\subseteq\cX$, then $L$ is called a $K$-moment functional.

Among the moment functionals the most important ones act on polynomials $\cV = \rset[x_1,\dots,x_n]$ on some $K\subseteq\rset^n$, $n\in\nset$. Here, the name \emph{moment} actually comes from. If $K$ is closed then Haviland's Theorem \cite{havila35,havila36} states that a linear functional $L:\rset[x_1,\dots,x_n]\to\rset$ is a $K$-moment functional if and only if $L(p)\geq 0$ for all $p\in\rset[x_1,\dots,x_n]$ with $p\geq 0$. On the other side $p\geq 0$ on $K$ if and only if $L(p)\geq 0$ for all $K$-moment functionals $L$ since every point evaluation is a moment functional. These are the two directions in the duality theorem and the many connections between the moment problem (deciding when a linear functional is a moment functional) and non-negative polynomials (and therefore optimization and many other applications) only start here. See e.g.\ \cite{ahiezer62}, \cite{akhiezClassical}, \cite{didioConeArXiv}, \cite{curto00}, \cite{fialkoMomProbSurv}, \cite{kreinMarkovMomentProblem}, \cite{lasserreSemiAlgOpt}, \cite{lauren09}, \cite{marshallPosPoly}, \cite{schmudMomentBook}, and references within for more on the moment problem, the connection to non-negative polynomials, and applications.

Besides the one-point evaluation $L(f) = f(x)$ the following is probably the simplest moment functional.

\begin{exm}\label{exm:lebOnUnit}
Let $\lambda$ be the Lebesgue measure on $[0,1]$ and let $\cV = \rset[t]$. Then the functional
\begin{equation}\label{eq:lebesgueMeasure}
L_{\text{Leb}}:\rset[t]\to\rset\quad\text{with}\quad L_{\text{Leb}}(t^d) = \int_0^1 t^d~\diff\lambda(t) = \frac{1}{d+1}\quad \text{for all}\ d\in\nset_0,
\end{equation}
is the unique linear functional such that $L(t^d) = \frac{1}{d+1}$ holds for all $d\in\nset_0$.\exmsym
\end{exm}

Besides this the general $[0,1]$-moment problem (also called the Hausdorff moment problem) is the easiest to decide.

\begin{thmHaus}[see \cite{hausdo21} or {\cite[Thm.\ 1.1 and 1.2]{kreinMarkovMomentProblem}}]\label{thm:haus}\
\begin{enumerate}[(a)]
\item Let $d\in\nset$. The following are equivalent.
\begin{enumerate}[(i)]
\item $L:\rset[x]_{\leq d}\to\rset$ is a $[0,1]$-moment functional.

\item $L(p)\geq 0$ holds for all $p\in\rset[x]_{\leq d}$ such that $p\geq 0$ on $[0,1]$.
\end{enumerate}

\item The following are equivalent.
\begin{enumerate}[(i)]
\item $L:\rset[x]\to\rset$ is a $[0,1]$-moment functional.

\item $L(p)\geq 0$ holds for all $p\in\rset[x]$ such that $p\geq 0$ on $[0,1]$.
\end{enumerate}
\end{enumerate}
\end{thmHaus}

This problem is fully solved since by the univariate Positivstellensatz every polynomial $p\in\rset[x]$ which is non-negative on $[0,1]$ has the form
\[p(x) = p_1(x)+ x\cdot(1-x)\cdot p_2(x) = q_1(x) + x\cdot q_2(x) + (1-x)\cdot q_3(x)\]
for some $p_i$, $q_i\in\sum\rset[x]^2$ sums of squares. This also holds with the degree bound $\deg p \leq d$.

In higher dimensions the problem is not completely solved and several problems appear, especially since in $\rset^n$ with $n\geq 2$ there are non-negative polynomials with are not sums of squares or a tuple $(X_1,\dots,X_n)$ of pairwise commuting and symmetric multiplication operators must have an extension to pairwise commuting and self-adjoint multiplication operators $(\overline{X_1},\dots,\overline{X_n})$.

To understand moment functionals better and to simplify them we investigate in this article the possibility of transforming a linear (moment) functional into another linear (moment) functional based on several isomorphism and transformation results between measure spaces. But before we give the formal definition of a transformation of a linear functional let us have a look at the following theorem to see what kind of results we are looking for.

\begin{thm}\label{thm:intro}
Let $S$ be a Souslin set (e.g.\ a Borel set $S\subseteq\rset^n$), $\cV$ be a vector space of real measurable functions $v:S\to\rset$, and $L:\cV\to\rset$ be a linear functional. Then the following are equivalent:
\begin{enumerate}[(i)]
\item $L:\cV\to\rset$ is a $S$-moment functional.

\item There exists a measurable function $f:[0,1]\to S$ such that
\begin{equation}\label{eq:momFctlRepr}
L(v) = \int_0^1 v(f(t))~\diff\lambda(t)
\end{equation}
for all $v\in\cV$ where $\lambda$ is the Lebesgue measure on $[0,1]$, i.e., $\lambda\circ f^{-1}$ is a representing measure of $L$.
\end{enumerate}
\end{thm}
\begin{proof}
(i)$\to$(ii): Let $\mu$ be a representing measure of $L$. By \Cref{prop:fromLebesgue} there exists a measurable function $f:[0,1]\to S$ such that $\mu = \lambda\circ f^{-1}$ and hence
\[L(v) = \int_S v(x)~\diff\mu(x) = \int_S v(x)~\diff(\lambda\circ f^{-1})(x) \overset{\text{\Cref{lem:integralTrans}}}{=} \int_0^1 v(f(t))~\diff\lambda(t) \tag{$*$}\]
for all $v\in\cV$.

(ii)$\to$(i): $\lambda\circ f^{-1}$ is a representing measure of $L$ by \Cref{lem:integralTrans}.
\end{proof}

\Cref{thm:intro} can be seen as a complete characterization of ($S$-)moment functionals, i.e., every moment functional $L:\cV\to\rset$ has the form (\ref{eq:momFctlRepr}) for some $f:[0,1]\to S$. Additionally, \Cref{thm:intro} also shows that every moment functional $L$ is represented by $\lambda\circ f^{-1}$ for a measurable function $f:[0,1]\to\rset^n$.

Hence, the aim of this paper is to characterize and represent moment functionals in the form of (\ref{eq:momFctlRepr}) and especially to find additional properties of $f:[0,1]\to S$.

The notation and result in \Cref{thm:intro} stimulate the notation of a transformation of a linear (moment) functional. We introduce the following definitions.

\begin{dfn}\label{dfn:transform}
Let $\cX$ and $\cY$ be two Souslin spaces, $\cU$ and $\cV$ two vector spaces of real measurable functions on $\cX$ resp.\ $\cY$, and $K:\cU\to\rset$ and $L:\cV\to\rset$ be two linear functionals. We say \emph{$L$ (continuously) transforms into $K$}, symbolized by $L\trans K$ resp.\ $L\ctrans K$, if there exists a Borel (resp.\ continuous) function $f:\cX\to\cY$ such that $\cV\circ f \subseteq\cU $ and $L(v) = K(v\circ f)$ for all $v\in\cV$.

We say \emph{$L$ strongly (and continuously) transforms into $K$}, symbolized by $L\strans K$ resp.\ $L\sctrans K$, if there exists a surjective Borel (resp.\ surjective and continuous) function $f:\cX\twoheadrightarrow\cY$ such that $\cV\circ f =\cU $ and $L(v) = K(v\circ f)$ for all $v\in\cV$.
\end{dfn}

If in this definition of a transformation $\trans$ a function $f:\cX\to\cY$ is fixed because it has special properties, then we denote that in the transformation by $\overset{f}{\trans}$. Of course, we have the implications
\begin{align*}
L\strans K \qquad&\Rightarrow\qquad L\trans K
\intertext{and}
L\sctrans K \qquad&\Rightarrow\qquad L\ctrans K \qquad\Rightarrow\qquad L\trans K.
\end{align*}
With this definition \Cref{thm:intro} can be reformulated to the following statement.

\begin{cor}\label{cor:intro}
$L:\cV\to\rset$ is a moment functional iff $L\leadsto [K:\cL^1([0,1],\lambda)\to\rset]$.
\end{cor}

The paper is structured as follows. In \Cref{sec:prel} we will give the preliminaries on measure theory and integration. Since most of the measure theoretic terminology and results in \Cref{sec:prel} (Souslin sets, Lebesgue--Rohlin spaces, isomorphisms between measure spaces etc.) have to our knowledge never been used in connection with the moment problem before, we give the complete definitions, results, and important examples which are essential for this paper (but without proofs).

In \Cref{sec:trans} we present basic properties of transformations (\Cref{dfn:transform}). E.g.\ in \Cref{thm:transMomProp} we show that if there exists a transformation $L\trans K$ and $K$ is a moment functional, then also $L$ is a moment functional. So, the transformation $\trans$ (literally and symbolically) aims at moment functionals $K$ to determine whether already $L$ was a moment functional.

\Cref{sec:main} contains then the main results where several non-trivial transformations to $[0,1]$- or $I_k$-moment functionals are presented, $I_k$ finite union of compact intervals in $\rset$. We show, which might already be apparent from \Cref{thm:intro}, that the structure of possible moment functionals $K$ are quite simple. These are always $[0,1]$- or $I_k$-moment functionals. However, this simplicity of $K$ has the price that $f:[0,1]\to S$ has little properties. In the worst case as in \Cref{thm:intro} we only have that $f$ is measurable. We therefore also present results where $f$ is at least continuous and can therefore approximated by polynomials on $[0,1]$ in the supremum norm.

In \Cref{sec:open} we give the conclusions and open problems. Additionally, we give and discuss several open questions, especially the restriction that $f$ is a rational or a polynomial map.

\section{Preliminaries: Measure Theory and the Lebesgue Integral}
\label{sec:prel}

We give here the measure theoretic results used in our paper. Of course, it is possible to go directly to \Cref{sec:trans} and the main results in \Cref{sec:main} and consult this \Cref{sec:prel} if necessary while reading the results and proofs.

In this article we follow the monographs \cite{federerGeomMeasTheo}, \cite{liebAnalysis}, and \cite{bogachevMeasureTheory} for the measure theory and Lebesgue integral. We denote by $\cP(\cX)$ the \emph{power set} of a set $\cX$, i.e., the set of all subsets of $\cX$. Let $\cA\subseteq\cP(\cX)$ be a $\sigma$-algebra on a set $\cX$, then we call $(\cX,\cA)$ a \emph{measurable space}. A function $f:(\cX,\cA)\to(\cY,\cB)$ between measurable spaces is called \emph{measurable} if $f^{-1}(B)\in\cA$ for all $B\in\cB$ holds.

Given $\cF\subseteq\cP(\cX)$, then by $\sigma(\cF)$ we denote the $\sigma$-algebra generated by $\cF$, i.e., the smallest $\sigma$-algebra containing $\cF$. The \emph{Borel $\sigma$-algebra} $\fB(\cX)$ of a topological (e.g.\ Hausdorff) space $\cX$ is generated by all open sets in $\cX$.

Given a measurable space $(\cX,\cA)$, a \emph{measure} $\mu$ on $(\cX,\cA)$ is a countably additive function $\mu:\cA\to [0,\infty]$. I.e., dissident from \cite{bogachevMeasureTheory} for us all measures are non-negative if not otherwise explicitly stated as \emph{signed}. $(\cX,\cA,\mu)$ is called a \emph{measure space}. $(\cX,\cA,\mu)$ is called \emph{probability measure space} if additionally $\mu(\cX)=1$. An \emph{atom} $\delta_x$ is a measure such that
\[\delta_x(A) = \begin{cases} 1 &\text{for}\ x\in A\\ 0 & \text{for}\ x\not\in A\end{cases}.\]
The special (Carath\'eodory) outer measures are used and treated in \Cref{sec:daniell}.

Let $\cX$ be a topological (e.g.\ locally compact Hausdorff) space. A measure on $(\cX,\fB(\cX))$ is called \emph{Borel measure}. A \emph{Radon measure} $\mu$ is a measure over $(\cX,\fB(\cX))$ such that $\mu(K)<\infty$ for all compact $K\subseteq\cX$ and $\mu(V) = \sup\{\mu(K) \,|\, K$ is compact, $K\subseteq V\}$. By $\lambda^n$ we denote the $n$-dimensional Lebegue measure on $(\rset^n,\fB(\rset^n))$.

Let $(\cX,\cA,\mu)$ be a measure space and $f:\cX\to [0,\infty]$ be a non-negative\linebreak measurable function. The \emph{Lebesgue integral} is defined by
\begin{equation}\label{eq:lebesgue}
\int_\cX f(x)~\diff\mu(x) := \int_0^\infty \mu(f^{-1}((t,\infty)))~\diff t
\end{equation}
since $h(t) := \mu(f^{-1}((t,\infty)))$ is non-increasing, i.e., Riemann integrable, with the Riemann integral $\int_0^\infty h(t)~\diff t$. $f$ is called $\mu$-integrable if (\ref{eq:lebesgue}) is finite. A general measurable function $f:\cX\to[-\infty,\infty]$ is called $\mu$-integrable if $f_+ := \max(f,0)$ and $f_- := -\min(f,0)$ are $\mu$-integrable. The Lebesgue integral is then defined by
\[\int_\cX f(x)~\diff\mu(x) := \int_\cX f_+(x)~\diff\mu(x) - \int_\cX f_-(x)~\diff\mu(x).\]
We have the following transformation formula.

\begin{lem}\label{lem:integralTrans}
Let $f:(\cY,\cB)\to(\rset,\fB(\rset))$ and $g:(\cX,\cA)\to (\cY,\cB)$ be measurable functions, $\mu$ a measure on $(\cX,\cA)$ such that $f\circ g$ is $\mu$-integrable. Then $\mu\circ g^{-1}$ is a measure on $(\cY,\cB)$ and $f$ is $\mu\circ g^{-1}$-integrable with
\begin{equation}\label{eq:integralTrans}
\int_\cX (f\circ g)(x)~\diff\mu(x) = \int_\cY f(y)~\diff(\mu\circ g^{-1})(y).
\end{equation}
\end{lem}
\begin{proof}
It is sufficient to show (\ref{eq:integralTrans}) for $f\geq 0$:
\begin{align*}
\int_\cX (f\circ g)(x)~\diff\mu(x) &= \int_0^\infty \mu((f\circ g)^{-1}((t,\infty)))~\diff t
&&= \int_0^\infty \mu(g^{-1}(f^{-1}((t,\infty))))~\diff t\\
&= \int_0^\infty (\mu\circ g^{-1})(f^{-1}((t,\infty)))~\diff t
\!\!\!\!\! &&= \int_\cY f(y)~\diff(\mu\circ g^{-1})(y).\qedhere
\end{align*}
\end{proof}

We have the first result from measure theory. We apply it in \Cref{prop:oneLebesgueDirection}.

\begin{prop}[see e.g.\ {\cite[Prop.\ 9.1.11]{bogachevMeasureTheory}}]\label{prop:makingLebesgue}
Let $\mu$ be an atomless probability measure on a measurable space $(\cX,\cA)$. Then there exists an $\cA$-measurable function $f:\cX\to [0,1]$ such that $\mu\circ f^{-1} = \lambda$ is the Lebesgue measure on $[0,1]$.
\end{prop}

The following is a central definition.

\begin{dfn}[{\cite[Def.\ 6.6.1]{bogachevMeasureTheory}}]\label{dfn:souslin}
A set in a Hausdorff space is called a \emph{Souslin set} if it is the image of a complete separable metric space under a continuous mapping. A \emph{Souslin space} is a Hausdorff space that is a Souslin set.
\end{dfn}

The empty set is a Souslin set. Souslin sets are fully characterized.

\begin{prop}[see e.g.\ {\cite[Prop.\ 6.6.3]{bogachevMeasureTheory}}]
Every non-empty Souslin set is the image of $[0,1]\setminus\qset$ under some continuous function and also the image of $(0,1)$ under some Borel mapping.
\end{prop}

More concrete examples which are important to us are the following.

\begin{exm}\label{exm:unitInterval}
The unit interval $[0,1]\subset\rset$ is of course a complete separable metric space (with the usual distance metric $d(x,y) := |x-y|$). The question which sets are the continuous images of $[0,1]$ is partially answered by \emph{space filling curves}, see e.g.\ \cite[Ch.\ 5]{saganSpaceFillingCurves}. So the \emph{Peano curves} as continuous and surjective functions
\[f:[0,1]\to [a_1,b_1]\times\cdots\times [a_n,b_n]\]
with $n\in\nset$ and $-\infty < a_i < b_i < \infty$ for all $i=1,\dots,n$ show that all hyper-rectangles are Souslin spaces/sets. Especially $[0,1]$ is a Souslin set/space.

A full answer gives the following theorem.\medskip

\begin{minipage}{\linewidth-1cm}
\begin{thmHahnM}[see e.g.\ {\cite[Thm.\ 6.8]{saganSpaceFillingCurves}}]\label{thm:hahnM}
A set $K$ in a non-empty Hausdorff space is the continuous image of $[0,1]$ if and only if it is compact, connected, and locally connected.
\end{thmHahnM}
\end{minipage}\medskip

\noindent
So sets $K\subseteq\rset^n$ are continuous images of $[0,1]$ if and only if they are compact and path-connected. Hahn--Mazurkiewicz also implies that $\pset\rset^n$ is a Souslin space.\exmsym
\end{exm}


\begin{lem}[see e.g.\ {\cite[Lem.\ 6.6.5, Thm.\ 6.6.6 and 6.7.3]{bogachevMeasureTheory}}]\label{lem:souslin}\
\begin{enumerate}[(i)]
\item The image of a Souslin set under a continuous function to a Hausdorff space is a Souslin set.

\item Every open or closed set of a Souslin space is Souslin.

\item If $A_n$ are Souslin sets in $\cX_n$ for all $n\in\nset$ then $\prod_{n\in\nset} A_n$ is a Souslin set in $\prod_{n\in\nset} \cX_n$.

\item If $A_n\subseteq\cX$ are Souslin sets in a Hausdorff space $\cX$, then $\bigcap_{n\in\nset} A_n$ and $\bigcup_{n\in\nset} A_n$ are Souslin sets.

\item Every Borel subset of a Souslin space is a Souslin space.

\item Let $A\subseteq\cX$ and $B\subseteq\cY$ be Souslin sets of Souslin spaces and $f:\cX\to\cY$ be a Borel function. Then $f(A)$ and $f^{-1}(B)$ are Souslin sets.
\end{enumerate}
\end{lem}

\begin{rem}\label{rem:nonBorelSouslin}
The reverse of \Cref{lem:souslin}(v) is in general not true. Not every Souslin set is Borel. In fact, every non-empty complete metric space without isolated points contains a non-Borel Souslin set, see e.g.\ \cite[Cor.\ 6.7.11]{bogachevMeasureTheory}.
\end{rem}

(vi) demonstrates the difference between Souslin sets and Borel sets (in $\rset^n$). While the continuous image of a Borel set is again a Borel set, this no longer holds for Borel functions. But as (vi) shows for the Souslin sets the preimage and image under measurable functions remain Souslin sets.

From \Cref{exm:unitInterval} and \Cref{lem:souslin} we get the following additional explicit \mbox{examples} of Souslin sets.

\begin{exm}
$\rset^n$ and every compact semi-algebraic set in $\rset^n$ (resp.\ $\pset\rset^n$) are Souslin sets.\exmsym
\end{exm}

\begin{dfn}
Let $(\cX,\cA)$ and $(\cY,\cB)$ be two measurable spaces. A measurable function $i:(\cX,\cA)\to (\cY,\cB)$ is called an \emph{isomorphism} and the two measurable spaces \emph{isomorphic} if $i$ is bijective, $i(\cA)=\cB$, and $i^{-1}(\cB)=\cA$.
\end{dfn}

The reason why we work with Souslin spaces is revealed in the following theorem.

\begin{thm}[see e.g.\ {\cite[Thm.\ 6.7.4]{bogachevMeasureTheory}}]\label{thm:isoOne}
Let $\cX$ be a Souslin space. Then there exist a Souslin set $S\subseteq [0,1]$ and an isomorphism $h:(S,\fB(S))\to (\cX,\fB(\cX))$.
\end{thm}

The existence of an isomorphism can be weakened. For Borel measurable function $f:\cX\to\cY$ between two Souslin spaces $\cX$ and $\cY$ with $f(\cX)=\cY$ one always finds nice (i.e., Borel measurable) one-sided inverse functions.

\begin{thmJankoff}[see e.g.\ {\cite[Thm.\ 6.9.1 and 9.1.3]{bogachevMeasureTheory}}]\label{thm:jankoff}
Let $\cX$ and $\cY$ be two Souslin spaces and let $f:\cX\to\cY$ be a surjective Borel mapping. Then there exists a Borel measurable function $g:\cY\to\cX$ such that $f(g(y))=y$ for all $y\in\cY$.
\end{thmJankoff}

In other words, restricting $f$ so some $\cX_0\subseteq\cX$ makes $\tilde{f}:=f|_{\cX_0}$ not only bijective but $\tilde{f}$ and $\tilde{f}^{-1}$ are measurable. We have
\[\cY\overset{g}{\to} \cX \overset{f}{\to} \cY \qquad\text{with}\qquad f\circ g = \id_\cY,\]
i.e., $g$ is injective, $f$ is surjective, and with $\cX_0 = \mathrm{im}\, g := g(\cY)$ we have $\tilde{f}^{-1} = g$.

\begin{dfn}[see e.g.\ {\cite[Def.\ 9.2.1]{bogachevMeasureTheory}}]
Let $(\cX,\cA,\mu)$ and $(\cY,\cB,\nu)$ be two measure spaces with non-negative measures.
\begin{enumerate}[i)]
\item A \emph{point isomorphism} $T:\cX\to\cY$ is a bijective mapping such that $T(\cA) = \cB$ and $\mu\circ T^{-1} = \nu$.

\item The spaces $(\cX,\cA,\mu)$ and $(\cY,\cB,\nu)$ are called \emph{isomorphic $\modzero$} if there exist sets $N\in\cA_\mu$, $M\in\cB_\nu$ with $\mu(N)=\nu(M)=0$ and a point isomorphism $T:\cX\setminus N\to\cY\setminus M$ that are equipped with the restriction of the measures $\mu$ and $\nu$ and the $\sigma$-algebras $\cA_\mu$ and $\cB_\nu$.
\end{enumerate}
\end{dfn}

A point isomorphism $T$ between $(\cX,\cA,\mu)$ and $(\cY,\cB,\nu)$ is of course measurable since $\nu(B) = (\mu\circ T^{-1})(B) = \mu(T^{-1}(B))$ implies $T^{-1}(B)\in\cA$ for all $B\in\cB$.

Like \Cref{thm:isoOne} also the next result shows the importance of working on Souslin sets.

\begin{thm}[see e.g.\ {\cite[Thm.\ 9.2.2]{bogachevMeasureTheory}}]\label{thm:isoSouslin}
Let $(\cX,\cA)$ be a Souslin space with Borel probability measure $\mu$. Then $(\cX,\cA,\mu)$ is isomorphic $\modzero$ to the space $([0,1],\fB([0,1]),\nu)$ for some $\nu$ Borel probability measure. If $\mu$ is an atomless measure, then one can take for $\nu$ the Lebesgue measure $\lambda$.
\end{thm}

\begin{cor}[see e.g.\ {\cite[Rem.\ 9.7.4]{bogachevMeasureTheory}}]\label{prop:fromLebesgue}
Let $\mu$ be a probability measure on a Souslin space $\cX$. Then there exists a measurable function $f:[0,1]\to\cX$ such that $\mu=\lambda\circ f^{-1}$ where $\lambda$ is the Lebesgue measure on $[0,1]$.
\end{cor}

For both results note the difference to \Cref{prop:makingLebesgue}. In \Cref{prop:makingLebesgue} we find for any measurable space $\cX$ and measure $\mu$ a map
\[f:\cX\to [0,1]\qquad \text{such that}\qquad \mu = \lambda\circ f^{-1}.\]
But for Souslin spaces $\cX$ in \Cref{prop:fromLebesgue} we find a map
\[f:[0,1]\to \cX\qquad \text{such that}\qquad \lambda = \mu\circ f^{-1}.\]
\Cref{thm:isoSouslin} restricts $f:[0,1]\to\cX$ to isomorphisms and hence not all measures can be transformed into $\lambda$. Atoms in the measure $\mu$ prevent it from being isomorphic to $\lambda$. In fact, as explained in \cite[Rem.\ 9.7.4]{bogachevMeasureTheory}, \Cref{prop:fromLebesgue} follows from \Cref{thm:isoSouslin} by introducing atoms into $f:[0,1]\to\cX$ by introducing constant functions into $f$.

But \Cref{thm:isoSouslin} provides that if $\mu$ has atoms, it can still be isomorphic $\modzero$ be transformed into a measure $\nu$ on $[0,1]$. Without atoms we could chose $\nu = \lambda$. So is it possible to transform the non-atomic part of $\mu$ to $\lambda$ and then add the atoms from $\mu$ to $\lambda$? Yes, we can. This is done on the following spaces.

\begin{dfn}[see e.g.\ {\cite[Def.\ 9.4.6]{bogachevMeasureTheory}}]
A measure space $(\cX,\cA,\mu)$ is called a \emph{Lebesgue--Rohlin space} if it is isomorphic $\modzero$ to some measure space $(\cY,\cB,\nu)$ with a countable basis with respect to which $\cY$ is complete.
\end{dfn}

\begin{exm}[see e.g.\ {\cite[Exm.\ 9.4.2]{bogachevMeasureTheory}}]\label{exm:lebRoh}
$(M,\fB(M),\mu)$, where $M$ is a Borel set of a complete separable metric space $\cX$ and $\mu$ is a Borel measure on $M$, is a Lebesgue--Rohlin space. Especially $\cX = \rset^n$ or $\pset\rset^n$ are complete metric spaces and therefore any Borel measure on a Borel subset $M\in\fB(\rset^n)$ gives a Lebesgue--Rohlin space.\exmsym
\end{exm}

We can now transform any measure by an isomorphism $\modzero$ to the Lebesgue measure $\lambda$ plus atoms.

\begin{thm}[see e.g.\ {\cite[Thm.\ 9.4.7]{bogachevMeasureTheory}}]\label{thm:lebRoh}
Let $(\cX,\cA,\mu)$ be a Lebesgue--Rohlin space with a probability measure $\mu$. Then it is isomorphic $\modzero$ to the interval $[0,1]$ with the measure $\nu = c\lambda + \sum_{i=1}^\infty c_n\cdot \delta_{1/n}$, where $c=1-\sum_{i=1}^\infty c_i$, $\mu(a_i) = c_i$ and $\{a_i\}\subseteq\cX$ is the family of all atoms of $\mu$.
\end{thm}

So we can transform any measure to the Lebesgue measure $\lambda$ on $[0,1]$ or to $\lambda$ on $[0,1]$ plus atoms. But these transformations are performed mainly by measurable functions because the set $\cX$ where the original measure lives it to large. If we restrict the space where the measure lives, we get better transformations, especially continuous ones.

\begin{thm}[see e.g.\ {\cite[Thm.\ 9.7.1]{bogachevMeasureTheory}}]\label{thm:lebIso}
Let $K$ be a compact metric space that is the image of $[0,1]$ under a continuous mapping $\tilde{f}$ and let $\mu$ be a Borel probability measure on $K$ such that $\supp\mu=K$. Then there exists a continuous and surjective mapping $f:[0,1]\to K$ such that $\mu = \lambda\circ f^{-1}$, $\lambda$ is the Lebesgue measure on $[0,1]$.
\end{thm}

We will apply \Cref{thm:lebIso} especially in connection with the \Cref{thm:hahnM}. The advantage is here that $f$ on $[0,1]$ is continuous and can therefore be approximated by polynomials up to any precision $\varepsilon>0$ in the $\sup$-norm.











\section{Transformations of linear Functionals: Basic Properties}
\label{sec:trans}

For the transformation $\leadsto$ in \Cref{dfn:transform} we get the following technical result.

\begin{lem}\label{lem:transShort}
Let $\cX$, $\cY$, and $\cZ$ be Souslin spaces; $\cU$, $\cV$, and $\cW$ be vector spaces of real measurable functions on $\cX$, $\cY$, and $\cZ$ respectively; and $M:\cW\to\rset$, $L:\cV\to\rset$, and $K:\cU\to\rset$ be linear functionals. The following hold:
\begin{enumerate}[(i)]
\item $M\trans L$ and $L\trans K$ imply $M\trans K$.

\item $M\ctrans L$ and $L\ctrans K$ imply $M\ctrans K$.

\item $M\strans L$ and $L\strans K$ imply $M\strans K$.

\item $M\sctrans L$ and $L\sctrans K$ imply $M\sctrans K$.
\end{enumerate}
\end{lem}
\begin{proof}
(i): Since $M\leadsto L$ there exists a Borel function $f:\cY\to\cZ$ such that $\cW\circ f\subseteq \cV$ and $M(w) = L(w\circ f)$ for all $w\in\cW$. And since $L\leadsto K$ there exists a Borel function $g:\cX\to\cY$ such that $\cV\circ g\subseteq \cU$ and $L(v) = K(v\circ g)$ for all $v\in\cV$. Hence, $h=f\circ g:\cX\to\cZ$ implies $\cW\circ h = \cW\circ f\circ g \subseteq \cV\circ g\subseteq\cU$ and $M(w) = L(w\circ f) = K(w\circ f\circ g) = K(w\circ h)$ for all $w\in\cW$, i.e., $M\leadsto K$.

(ii)-(iv) follow in the same way as (i).
\end{proof}

\Cref{lem:transShort} can be seen as shortening the sequence:
\[M\leadsto L\leadsto K \qquad\Rightarrow\qquad M\leadsto K.\]

The next lemma shows, that a strong transformation $L\strans K$ implies the reverse transformation $K\trans L$.

\begin{lem}\label{lem:backTrans}
Let $\cX$ and $\cY$ be Souslin sets, $\cU$ and $\cV$ vector spaces of real functions on $\cX$ resp.\ $\cY$, and $L:\cV\to\rset$ and $K:\cU\to\rset$ be linear functionals. Then $L\strans K$ implies $K\trans L$.
\end{lem}
\begin{proof}
Since $L\strans K$ there exists a surjective Borel function $f:\cX\to\cY$ such that $L(v) = K(v\circ f)$ and $\cV\circ f = \cU$. Since $f$ is surjective by \Cref{thm:jankoff} there exists a Borel function $g:\cY\to\cX$ such that $f(g(y))=y$ for all $y\in\cY$. Let $u\in\cU = \cV\circ f$, then $v$ in $u = v\circ f$ is unique since for $v_1$ and $v_2$ with that property we have
\[v_1 = v_1\circ f\circ g = u\circ g = v_2\circ f\circ g = v_2.\]
Hence, $\cU\circ g = \cV$ and for all $u\in\cU$ we have
\[K(u) = K(v\circ f) = L(v) = L(v\circ f\circ g) = L(u\circ g).\qedhere\]
\end{proof}

While we have so far only transformed linear functionals, the importance of the transformation is revealed in the following result. It shows that the property of being a moment functional is preserved in one or both directions.

\begin{thm}\label{thm:transMomProp}
Let $\cX$ and $\cY$ be Souslin sets, $\cU$ and $\cV$ vector spaces of real functions on $\cX$ resp.\ $\cY$, and $L\!:\cV\to\rset$ and $K\!:\cU\to\rset$ be linear functionals. If $L\trans K$, then
\begin{enumerate}[(i)]
\item $K$ is a moment functional
\end{enumerate}
implies
\begin{enumerate}[(i)]\setcounter{enumi}{1}
\item $L$ is a moment functional.
\end{enumerate}
If $L\strans K$, then (i) $\Leftrightarrow$ (ii).
\end{thm}
\begin{proof}
Since $L\leadsto K$ there exists a Borel function $f:\cX\to\cY$ such that $\cV\circ f\subseteq \cU$ and $L(v) = K(v\circ f)$ for all $v\in\cV$.

(i)$\to$(ii): Let $K$ be a moment functional with representing measure $\nu$ on $\cX$, then
\[L(v) = K(v\circ f) = \int_\cX (v\circ f)(x)~\diff\nu(x) \overset{\text{\Cref{lem:integralTrans}}}{=} \int_\cY v(y)~\diff(\nu\circ f^{-1})(y),\]
i.e., $\nu\circ f^{-1}$ is a representing measure of $L$ and hence $L$ is a moment functional.

(ii)$\to$(i): When $L\strans K$, then \Cref{lem:backTrans} implies $K\trans L$.
\end{proof}

The importance of the transformation and hence \Cref{thm:transMomProp} can be seen in
\begin{equation}\label{eq:transGraph}
\begin{matrix}
&& L_8 && L_6 &\trans&L_5\\
&&\downtrans&&&&\downtrans\\
L_4 &\trans & L_3 & \trans & L_2 & \trans & L_1 &\trans & K\\
&&&&\uptrans\\
&&&& L_7
\end{matrix}.
\end{equation}
If $K$ is a moment functional, then all $L_1,\dots, L_8$ are moment funtionals. Assume in (\ref{eq:transGraph}) all transformations $\trans$ are strong transformations $\strans$. Then: If one $L_i$ or $K$ is a moment functional, then all $K,L_1,\dots,L_8$ are moment functionals.

Note, the transformation $\trans$ in \Cref{dfn:transform} also covers extensions and restrictions of functionals. Let $f = \id_\cX$ and let $\cV$ be a vector space of measurable functions on $\cX$, $\cV_0\subseteq\cV$ be a linear subspace, and $L:\cV\to\rset$ a linear functional. Then
\[L|_{\cV_0}\overset{\id_\cX}{\trans} L.\]
Or if $L_i:\cV_i\to\rset$ are extensions of $L$, i.e., $\cV\subseteq\cV_1\subseteq\cV_2\subseteq\dots\subseteq\cV_k$ with $L_i = L_{i+1}|_{\cV_i}$, then
\[L\overset{\id_\cX}{\trans} L_1 \overset{\id_\cX}{\trans} L_2 \overset{\id_\cX}{\trans}\dots \overset{\id_\cX}{\trans} L_k
\qquad\text{or short}\qquad
L\trans L_1\trans L_2\trans\dots\trans L_k\]
shows that if $L_k$ is a moment functional, then all $L_i$ and $L$ are moment functionals.

So far we introduced the transformation of a linear functional and gained basic properties. But as seen from \Cref{thm:intro} and \Cref{cor:intro}, there are non-trivial results for the transformations. The next section is devoted to these non-trivial transformation results.

\section{Non-trivial Transformations of linear Functionals}
\label{sec:main}

Let $\cV$ be (finite or infinite dimensional) vector space of measurable functions on a Souslin space $\cX$. Then by \Cref{thm:isoOne} there exist a Souslin set $S\subseteq [0,1]$ and an isomorphism $h:(S,\fB(S))\to(\cX,\fB(\cX))$. This implies that $\tilde{L}:\tilde{\cV}\to\rset$ with $\tilde{\cV}:=\{f\circ h \,|\, f\in\cV\}$ and $\tilde{L}(g) := L(g\circ h^{-1})$, $g\in\tilde{\cV}$, is a linear functional but now the functions $\tilde{\cV}$ live on $S\subseteq [0,1]$. Especially, $L$ is a moment functional if and only if $\tilde{L}$ is a moment functional.

For example, let $L:\rset[x_1,\dots,x_n]\to\rset$ be a moment functional with $\cX=\rset^n$. Then $h=(h_1,\dots,h_n):S\subseteq [0,1]\to\rset^n$ is an isomorphism between $(S,\fB(S))$ and $(\rset^n,\fB(\rset^n))$ and $\tilde{L}$ is a moment functional with $\tilde{L}(h^\alpha) = L(x^\alpha)$.

However, by \Cref{rem:nonBorelSouslin} $S$ needs not to be a Borel set. So determining whether $\tilde{L}$ is a moment functional might be as hard as determining whether $L$ is a moment functional. Additionally, $\tilde{L}$ now no longer lives on polynomials but evaluates measurable functions $h^\alpha = h_1^{\alpha_1}\cdots h_n^{\alpha_n}$ with $\alpha=(\alpha_1,\dots,\alpha_n)\in\nset_0^n$.

Allowing general Borel measurable functions on measurable spaces instead of isomorphisms we get \Cref{thm:intro} in the introduction. There we showed that any moment functional can be expressed as integration with respect to the Lebesgue measure $\lambda$ on $[0,1]$.

The next result shows that any moment functional with an atomless representing measure has a ``direction'' in which it looks like (\ref{eq:lebesgueMeasure}), i.e, the Lebesgue measure on $[0,1]$ evaluated on $\rset[t]$.

\begin{prop}\label{prop:oneLebesgueDirection}
Let $\cV$ be a vector space of real measurable functions on a measurable space $(\cX,\cA)$ such that there exists an element $v\in\cV$ with $1\leq v$ on $\cX$ and let $L:\cV\to\rset$ be a moment functional which has an atomless representing measure. Then there exists a measurable function $f:\cX\to [0,1]$ and an extension $\overline{L}:\cV+\rset[f]\to\rset$ of $L$ such that $\overline{L}(f^d) = \frac{\overline{L}(1)}{d+1}$ for all $d\in\nset_0$, i.e., $\tilde{L}:\rset[t]\to\rset$ with $\tilde{L}(t^d) := \overline{L}(f^d)$ for all $d\in\nset_0$ is represented by $\overline{L}(1)\cdot\lambda$ where $\lambda$ is the Lebesgue measure $\lambda$ on $[0,1]$.
\end{prop}
\begin{proof}
Let $\mu$ be a representing measure of $L$. By \Cref{prop:makingLebesgue} there exists a measurable $f:\rset^n\to [0,1]$ such that $\mu\circ f^{-1} = \lambda$ on $[0,1]$. Since $f$ is measurable, $|f|\leq 1$ on $\rset^n$, and $L(1)<\infty$, all $f^d$, $d\in\nset_0$, are $\mu$-integrable:
\[\left|\int_{\rset^n} f^d(x)~\diff\mu(x)\right| \leq \int_{\rset^n} |f(x)|^d~\diff\mu(x)\leq \int_{\rset^n} 1~\diff\mu(x) = L(1).\]
Define $\overline{L}:\rset[f]\to\rset$ by $\overline{L}(f^d) := \int_{\rset^n} f^d(x)~\diff\mu(x)$. Then
\[\overline{L}(f^d) = \int_{\rset^n} f^d(x)~\diff\mu(x) \overset{\text{\Cref{lem:integralTrans}}}{=} \int_0^1 t^d~\diff(\mu\circ f^{-1})(t) = \int_0^1 t^d~\diff\lambda(t) = \frac{L(1)}{d+1}\]
is represented by $L(1)\cdot\lambda$ on $[0,1]$.
\end{proof}

Hence, for any moment functional with an atomless representing measure there exists a function $f$ (a direction) such that it acts on $\rset[f]\cong\rset[t]$ as (\ref{eq:lebesgueMeasure}), i.e., the Lebesgue measure on $[0,1]$. Atomless representing measures are very common in the truncated moment problem. Under some mild conditions every truncated moment functional in the interior of the truncated moment cone has an atomless representing measure. We can even find a linear combination of Gaussian distributions (Gaussian mixture) as a representing measure. This was proven in \cite{didio18gaussian} for the first time.

Using the transformation $\trans$ formulation with $L_{\text{Leb}}$ from \Cref{exm:lebOnUnit} we can visualize \Cref{prop:oneLebesgueDirection} as
\[\begin{matrix}
&& L:\cV\to\rset\\
&& \downtrans\, {\scriptstyle\id_\cX}\\
L_{\text{Leb}}:\rset[t]\to\rset & \overset{f}{\trans} & \overline{L}:\cV+\rset[f]\to\rset.
\end{matrix}\]

Note the reverse statement of \Cref{prop:oneLebesgueDirection}. If a linear functional $L$ can never be (continuously) extended to $\rset[f]$ with $\overline{L}(f^d) = \frac{\overline{L}(1)}{d+1}$ for some measurable $f$, then $L$ is not a moment functional with an atomless representing measure.

\Cref{thm:intro} and \Cref{prop:oneLebesgueDirection} are very general. Especially \Cref{thm:intro} works on arbitrary Borel sets of $\rset^n$ (in fact on every Souslin space). For this generality we have to pay the price that $f$ is in general only measurable. Additionally, since we always express $L$ as integration with respect to $\lambda$ on $[0,1]$, the chosen $f$ depends on $L$. If we want additional properties for $f$ to hold, especially continuity and independence from $L$, then we need to restrict the functionals we want to transform. This can be achieved by restricting the investigation to $K$-moment functionals on compact and path-connected sets $K\subset\rset^n$. Then from the \Cref{thm:hahnM} we get the existence of surjective and continuous functions $f:[0,1]\to\ K$. We find the following result.

\begin{thm}\label{thm:first}
Let $n\in\nset$ be a natural number, $K\subset\rset^n$ be a compact and path-connected set, and let $\cV$ be a vector space of real measurable functions on $(K,\fB(K))$. Then any surjective and continuous function $f:[0,1]\to K$ induces for any linear functional $L:\cV\to\rset$ a strong and continuous transformation
\[L:\cV\to\rset \quad\overset{sc:f}{\trans}\quad \tilde{L}:\cV\circ f\to\rset,\]
i.e., for any linear functional $L:\cV\to\rset$ the following are equivalent:
\begin{enumerate}[(i)]
\item $L:\cV\to\rset$ is a $K$-moment functional.

\item $\tilde{L}:\cV\circ f\to\rset$ defined by $\tilde{L}(v\circ f) := L(v)$ is a $[0,1]$-moment functional.
\end{enumerate}
If $\tilde{\mu}$ is a representing measure of $\tilde{L}$, then $\tilde{\mu}\circ f^{-1}$ is a representing measure of $L$.

There exists a measurable function $g:K\to [0,1]$ such that $f(g(x))=x$ for all $x\in K$ and if $\mu$ is a representing measure of $L$, then $\mu\circ g^{-1}$ is a representing measure of $\tilde{L}$.
\end{thm}
\begin{proof}
Since $K\subset\rset^n$ is compact and path-connected, by the \Cref{thm:hahnM} there exists a continuous and surjective function $f:[0,1]\to K$. By \Cref{exm:unitInterval} or \Cref{lem:souslin} $[0,1]$ and $K$ are Souslin spaces and $f$ is Borel measurable (since it is continuous). By \Cref{thm:jankoff} there exists a measurable function $g:K\to [0,1]$ such that
\[f(g(x))=x\quad \text{for all}\ x\in K.\tag{$*$}\]
$(*)$ implies that $\tilde{L}$ is well-defined by $\tilde{L}(v\circ f) = L(v)$. To show this, for $\tilde{v}\in\tilde{\cV}$ let $v_1,v_2\in\cV$ be such that $v_1\circ f = \tilde{v} = v_2\circ f$. But then $g$ resp.\ $(*)$ implies $v_1 = v_1\circ f\circ g = \tilde{v}\circ g = v_2\circ f\circ g = v_2$, i.e., for any $\tilde{v}\in\cV$ there is a unique $v\in\cV$ with $\tilde{v} = v\circ f$.

(i)$\to$(ii): Let $L:\cV\to\rset$ be a $K$-moment functional and $\mu$ be a representing measure of $L$, i.e., $\supp\mu\subseteq K$ and
\[L(v) = \int_K v(x)~\diff\mu(x) \quad\text{for all}\ v\in\cV.\]
Then
\begin{align*}
\tilde{L}(v\circ f) = L(v) = \int_K v(x)~\diff\mu(x)
&\overset{\phantom{\text{\Cref{lem:integralTrans}}}}{=} \int_K (v\circ f)(g(x))~\diff\mu(x)\\
&\overset{\text{\Cref{lem:integralTrans}}}{=} \int_0^1 (v\circ f)(y)~\diff(\mu\circ g^{-1})(y),
\end{align*}
i.e., $\mu\circ g^{-1}$ is a representing measure of $\tilde{L}$ and hence $\tilde{L}$ is a $[0,1]$-moment functional.

(ii)$\to$(i): Let $\tilde{\mu}$ be a representing measure of $\tilde{L}:\tilde{\cV}\to\rset$. Then
\begin{align*}
L(v) = \tilde{L}(v\circ f) = \int_0^1 (v\circ f)(y)~\diff\tilde{\mu}(y) \overset{\text{\Cref{lem:integralTrans}}}{=} \int_K v(x)~\diff(\tilde{\mu}\circ f^{-1})(x),
\end{align*}
i.e., $\tilde{\mu}\circ f^{-1}$ is a representing measure of $L$ with $\supp\tilde{\mu}\circ f^{-1}\subseteq K$ and $L$ is therefore a $K$-moment sequence.
\end{proof}

In the previous result the functions $f:[0,1]\to K$ and $g:K\to [0,1]$ do not depend on the functions $\cV$ or the functional $L:\cV\to\rset$. They depend only on $K$. We can therefore fix such functions $f$ and $g$ and investigate any $L$ resp.\ $\tilde{L}$.

If the continuous $f$ can be chosen for each $L$, then in \Cref{thm:first}(ii) we can even ensure that $\tilde{L}$ is represented by the Lebesgue measure $\lambda$ on $[0,1]$ if and only if $L$ has a representing measure $\mu$ with $\supp\mu = K$, see \Cref{thm:contFLebesgue} below.

In \Cref{thm:first} we required that $K$ consists of one path-connected component. If $K$ consists of more than one component, then we can glue the parts together.

\begin{cor}\label{cor:first}
Let $n\in\nset$ and $K\subset\rset^n$ be the union of $k\in\nset\cup\{\infty\}$ compact, path-connected and pairwise disjoint sets $K_i\subset\rset^n$: $K = \bigcup_{i=1}^k K_i$. Let $\cV$ be a vector space of real valued measurable functions on $(K,\fB(K))$. There exists a continuous surjective function
\[f:\bigcup_{i=1}^k [2i-2,2i-1]\to K\]
such that for any linear functional $L:\cV\to\rset$ the following are equivalent:
\begin{enumerate}[(i)]
\item $L:\cV\to\rset$ is a $K$-moment functional.

\item $\tilde{L}:\tilde{\cV}\to\rset$ on $\tilde{\cV}:= \{v\circ f \,|\, v\in\cV\}$ and defined by $\tilde{L}(v\circ f) := L(v)$ is a $\bigcup_{i=1}^k [2i-2,2i-1]$-moment functional.
\end{enumerate}
\end{cor}
\begin{proof}
It is sufficient to show the existence of the function $f$ (and $g$). The rest of the proof is verbatim the same as in the proof of \Cref{thm:first}.

Since for each $i=1,2,\dots,k$ the set $K_i$ is compact and path-connected and the translation of the unit interval $[0,1]$ to $[2i-2,2i-1]$ is continuous, by the \Cref{thm:hahnM} there exists a continuous and surjective $f_i:[2i-2,2i-1]\to K_i$. Define $f:\bigcup_{i=1}^k [2i-2,2i-1]\to K$ by $f(x) := f_i(x)$ if $x\in [2i-2,2i-1]$ for an $i\in\{1,2,\dots,k\}$. Then $f$ is continuous and surjective.

For $g:K\to \bigcup_{i=1}^k [2i-2,2i-1]$ we proceed in the same way. By \Cref{thm:jankoff} for each $f_i:[2i-2,2i-1]\to K_i$ there exists a measurable $g_i: K_i\to [2i-2,2i-1]$. Hence, we define $g$ as $g(x) := g_i(x)$ if $x\in K_i$.
\end{proof}

Note, that when $K$ consists of countably many compact and path-connected components ($k=\infty$), then in \Cref{cor:first} $f$ is no longer supported on a bounded (and therefore compact) set: $\bigcup_{i=1}^k [2i-2,2i-1]$. But if e.g.\ $K$ is a compact and semi-algebraic set, then $K$ has only finitely many path-connected components.

An advantage in \Cref{thm:first} is that $f=(f_1,\dots,f_n):[0,1]\to K\subset\rset^n$ is continuous. Hence, all coordinate functions $f_i:[0,1]\to\rset$ are continuous. By the Stone--Weierstrass Theorem we can approximate each $f_i$ in the $\sup$-norm on $[0,1]$ by polynomials to any precision. $f$ can therefore be approximated to any precision by a polynomial map. A representing measure $\tilde{\mu}$ of $\tilde{L}$ provides the representing measure $\tilde{\mu}\circ f^{-1}$ of $L$. An approximation $f_\varepsilon\in\rset[x_1,\dots,x_n]^n$ of $f$, i.e., $\sup_{t\in [0,1]} \|f(t) - f_\varepsilon(t)\|<\varepsilon$ with any (fixed) norm $\|\,\cdot\,\|$ on $\rset^n$ and $\varepsilon>0$, provides an approximate representing measure $\tilde{\mu}\circ f_\varepsilon^{-1}$ of $L$.

Let $K\subset\rset^n$ be a compact and path-connected set, $\cV = \rset[x_1,\dots,x_n]$, and $L:\cV\to\rset$ be a linear functional. Then the induced functional $\tilde{L}:\tilde{\cV}\to\rset$ on $[0,1]$ is defined by $\tilde{L}(p\circ f):= L(p)$. It depends on $p\circ f$, i.e., $f^\alpha = f_1^{\alpha_1}\cdots f_n^{\alpha_n}$, $\alpha=(\alpha_1,\dots,\alpha_n)\in\nset_0^n$. So as in \Cref{thm:intro} the algebraic structure of $\rset[x_1,\dots,x_n]$ remains but the domain $K$ is pulled back to $[0,1]$ by the continuous $f$.

That the algebraic structure remains also reveals one big difference between $L$ and $\tilde{L}$. E.g.\ $\cV=\rset[x_1,\dots,x_n]$ separates points and is therefore dense in $C(K,\rset)$. But $f:[0,1]\to K$ is a space filling curve and therefore never injective (Netto's Theorem). Hence, there are $t_1,t_2\in[0,1]$ with $t_1\neq t_2$ and $f(t_1)=f(t_2)$. The set $\tilde{\cV} := \{p\circ f\,|\, p\in\cV\}$ therefore does not separate $t_1$ from $t_2$ and is by the Stone--Weierstrass Theorem not dense in $C([0,1],\rset)$. So the $\tilde{L}$ in \Cref{thm:first} and \Cref{cor:first} can at this point not extended to the \Cref{thm:haus}.

In the next theorem we will identify each $K$-moment functional with a $[0,1]$-moment functional, i.e., the \Cref{thm:haus}.

\begin{thm}\label{thm:second}
Let $n\in\nset$ be a natural number and $K\subset\rset^n$ be a compact and path-connected set. Then there exists a measurable function
\[g:K\to [0,1]\]
such that for all linear functionals $L:\cV\to\rset$ with $1\in\cV\subseteq C(K,\rset)$ the following are equivalent:
\begin{enumerate}[(i)]
\item $L:\cV\to\rset$ is a $K$-moment functional.

\item $L:\cV\to\rset$ continuously\footnote{If $p_i\in\rset[t]$ with $p_i\rightrightarrows p\in C([0,1],\rset)$ and $p\circ g\in\cV$ then $\overline{L}(p_i\circ g)\to L(p\circ g)$.} extends to $\overline{L}:\cV+\rset[g]\to\rset$ such that $\tilde{L}:\rset[t]\to\rset$ defined by $\tilde{L}(t^d):=\overline{L}(g^d)$ for all $d\in\nset_0$ is a $[0,1]$-moment functional, i.e.,
\begin{equation}\label{eq:directionTrans}
\begin{matrix}
&& L:\cV\to\rset\\
&& \downtrans\, {\scriptstyle\id_\cX}\\
\tilde{L}:\rset[t]\to\rset & \overset{g}{\trans} & \overline{L}:\cV+\rset[g]\to\rset.
\end{matrix}
\end{equation}
\end{enumerate}
If $\mu$ is the representing measure of $L$, then $\mu\circ g^{-1}$ represents $\tilde{L}$.

Additionally, there exists a continuous and surjective function $f:[0,1]\to K$ independent on $L$ resp.\ $\tilde{L}$ such that $f(g(x)) = x$ for all $x\in K$ and if $\tilde{\mu}$ is the representing measure of $\tilde{L}$, then $\tilde{\mu}\circ f^{-1}$ is the representing measure of $L$.
\end{thm}
\begin{proof}
Since $K$ is a compact and path-connected set, by the \Cref{thm:hahnM} there exists a continuous and surjective function $f:[0,1]\to K$. By \Cref{lem:souslin} $[0,1]$ and $K$ are Souslin sets and hence by \Cref{thm:jankoff} there exists a measurable function $g:K\to [0,1]$ such that
\[f(g(x))=x\quad \text{for all}\quad x\in K.\tag{$*$}\]

(i)$\to$(ii): Let $L:\cV\to\rset$ be a $K$-moment functional and $\mu$ be a representing measure of $L$ with $\supp\mu\subseteq K$. $g$ is measurable with $|g|\leq 1$ and hence we have that all $g^d$, $d\in\nset_0$, are $\mu$-integrable by
\[\left|\int_K g(x)^d~\diff\mu(x)\right|\leq\int_K |g(x)|^d~\diff\mu(x)\leq\int_K 1~\diff\mu(x) =\mu(K) = L(1)\tag{\$}\]
and hence $L$ extents to $\rset[g]$. Let $p\in\rset[t]$, then
\[\tilde{L}(p) = L(p\circ g) = \int_K (p\circ g)(x)~\diff\mu(x) \overset{\text{\Cref{lem:integralTrans}}}{=} \int_0^1 p(t)~\diff(\mu\circ g^{-1})(t)\]
and $\mu\circ g^{-1}$ is a representing measure of $\tilde{L}$, i.e., $\tilde{L}$ is a $[0,1]$-moment functional.

(ii)$\to$(i): Let $\tilde{L}:\rset[t]\to\rset$ be a $[0,1]$-moment functional and $\tilde{\mu}$ be its unique representing measure. Since by the Stone--Weierstrass Theorem $\rset[t]$ is dense in $C([0,1],\rset)$ the moment functional $\tilde{L}$ extends uniquely to $C([0,1],\rset)$. For simplicity we denote this extension also $\tilde{L}:C([0,1],\rset)\to\rset$. Since $f:[0,1]\to K$ is continuous we have $v\circ f\in C([0,1],\rset)$ for all $v\in\cV$. By ($*$) we have $v = v\circ f\circ g$ for all $v\in\cV$ and hence
\[L(v) = L(v\circ f\circ g).\tag{\&}\]
But since $v\circ f:[0,1]\to\rset$ is continuous and $\tilde{L}:\rset[t]\to\rset$ uniquely extends to $C([0,1],\rset)$ we have
\[L(v\circ f\circ g) = \tilde{L}(v\circ f).\tag{\#}\]
In summary we get
\[L(v) \overset{(\&)}{=} L(v\circ f\circ g) \overset{(\#)}{=} \tilde{L}(v\circ f) = \int_0^1 (v\circ f)(t)~\diff\tilde{\mu}(t)\overset{\text{Lem.\ \ref{lem:integralTrans}}}{=} \int_K v(x)~\diff(\tilde{\mu}\circ f^{-1})(x)\]
for all $v\in\cV$, i.e., $\tilde{\mu}\circ f^{-1}$ is a representing measure of $L$ and $L$ is therefore a $K$-moment functional.
\end{proof}

We see that all about $L$ is already known if we know how it acts (via $\tilde{L}$) on powers of the fixed (and independent on $L$) function $g$. $\tilde{L}:\rset[t]\to\rset$ is only a Hausdorff moment problem and its representing measure $\tilde{\mu}$ provides a representing measure $\mu = \tilde{\mu}\circ f^{-1}$ via a fixed (and independent on $L$) continuous function $f$.

\begin{rem}
Note, that in \Cref{thm:second} and therefore also in \Cref{cor:second} the condition $1\in\cV$ can be weakened to:
\[\text{There shall exists a}\ v\in\cV\subseteq C(K,\rset)\ \text{such that}\ v>0\ \text{on}\ K.\]
By compactness of $K$ and continuity of $v$ this implies $1\leq c\cdot v\in\cV$ for some $c>0$, i.e., $\mu(K)<\infty$ in (\$). However, since we have to extend $L:\cV\to\rset$ to $\overline{L}:\cV+\rset[g]\to\rset$ and $1\in\rset[g]$ we can assume w.l.o.g.\ already $1\in\cV$. If $1\not\in\cV$ and $L$ can not be extended to $1$, then $L$ can definitely not be extended to $\rset[g]$ and the statements of \Cref{thm:second} and \Cref{cor:second} remain valid.\exmsym
\end{rem}

\Cref{thm:second} requires the existence of a continuous extension $\overline{L}:\cV+\rset[g]\to\rset$ of $L$. Under the very mild condition $1\in\cV$ (resp.\ $v\in\cV$ with $v > 0$ on $K$ by the previous remark) extensions (not necessarily continuous) exist.

\begin{lem}\label{lem:extension}
Let $g$ be as in \Cref{thm:second} (resp.\ \Cref{cor:second}) and $L:\cV\to\rset$ be a linear functional on the vector space $\cV$ with $1\in\cV\subseteq C(K,I_k)$ and $L(1)>0$. Then there exists an extension $\overline{L}:\cV+\rset[g]\to\rset$ of $L:\cV\to\rset$.
\end{lem}
\begin{proof}
Since $g:K\to I_k\subseteq [0,1]$ in \Cref{thm:second} (resp.\ \Cref{cor:second}) we have $|g|\leq 1$. Hence, $1\in\cV\cap\rset[g]\neq\emptyset$ and $\cV+\rset[g] = \cV\oplus (\rset[g]\setminus\cV)$, i.e., $f=f_1+f_2\in\cV+\rset[g]$ with unique $f_1\in\cV$ and $f_2\in\rset[g]\setminus\cV$. Define
\[p:\cV+\rset[g]\to\rset \qquad\text{by}\qquad p(f) := |L(f_1)| + L(1)\cdot\|f_2\|_\infty\]
for all $f=f_1 + f_2\in\cV+\rset[g]$, $f_1\in\cV$, and $f_2\in\rset[g]\setminus\cV$. Hence, $L(f)\leq p(f)$ for all $f\in\cV$. Then
\[
p(f+g)\leq p(f)+p(g) \qquad\text{and}\qquad p(\alpha\cdot f) = \alpha\cdot p(f)\]
hold for all $f,g\in\cV+\rset[g]$ and $\alpha\geq 0$. By the Hahn--Banach Theorem there exists an extension $\overline{L}:\cV+\rset[g]\to\rset$ of $L$.
\end{proof}

An extension $\overline{L}$ in \Cref{lem:extension} is in general not unique. If $\cV$ is a point separating algebra on $K$ and $L$ is a $K$-moment functional, then the extension $L$ is unique (and continuous), since then the representing measure $\mu$ of $L$ is unique.

For the extension $\overline{L}$ it is only necessary that $1\in\cV$ to ensure $|g|\leq 1\in\cV$. $\cV\subseteq C(K,I_k)$ continuous is actually not necessary and hence \Cref{lem:extension} can be easily weakened.

As in \Cref{thm:first} also in \Cref{thm:second} the functions $f$ and $g$ do not depend on $L$ or $\tilde{L}$. They depend only on $K$. And as in \Cref{prop:oneLebesgueDirection} the functional $\tilde{L}$ is defined in one ``direction'' $\rset[g]\cong\rset[t]$ by $\tilde{L}(t^d) :=\overline{L}(g^d)$. But now it no longer needs to be $L_{\text{Leb}}$ as in \Cref{exm:lebOnUnit}.


The problem of determining whether $\tilde{L}:\rset[t]\to\rset$ in \Cref{thm:second}(ii) is a $[0,1]$-moment functional is the \Cref{thm:haus}. This problem is fully solved, analytically as well as numerically. But the function $g:K\to [0,1]$ to establish the equivalence (i) $\Leftrightarrow$ (ii) in \Cref{thm:second} is a measurable function and not a polynomial. Hence, $\overline{L}(g^d)$ is not directly accessible unless of course $d=0$. Fortunately, since $K\subset\rset$ is compact, $\rset[x_1,\dots,x_n]$ is dense in $C(K,\rset)$. Hence, for any given finite measure $\mu$ on $K$, i.e., $\mu(K)=L(1)<\infty$, we can approximate $g$ by a polynomial $g_\varepsilon\in\rset[x_1,\dots,x_n]$ in the $L^1(\mu)$-norm to any arbitrary precision.

\begin{thm}\label{thm:secondApprox}
Let $n\in\nset$ be a natural number, $K\subset\rset^n$ be a compact and path-connected set, and let $g:K\to [0,1]$ be from \Cref{thm:second}. Then for any $\varepsilon>0$ and $K$-moment functional $L:\rset[x_1,\dots,x_n]\to\rset$ there exists a polynomial \mbox{$g_\varepsilon\in\rset[x_1,\dots,x_n]$} such that
\[L(|g_\varepsilon-g|) \leq \varepsilon
\qquad\text{and}\qquad
|L(g^d) - L(g_\varepsilon^d)| \leq d\cdot L(|g-g_\varepsilon|) \leq d \cdot \varepsilon\]
hold for all $d\in\nset_0$. $g_\varepsilon$ can be chosen to be a square: $g_\varepsilon = p_\varepsilon^2$, $p_\varepsilon\in\rset[x_1,\dots,x_n]$.
\end{thm}
\begin{proof}
$L$ is a $K$-moment functional and therefore has a unique representing measure $\mu$ with $\supp\mu\subseteq K$. $g\geq 0$ and hence there exists a measurable function $p:K\to [0,1]$ such that $g=p^2$. Since $K$ is compact and $\mu(K) = L(1) < \infty$ the polynomials $\rset[x_1,\dots,x_n]$ are dense in $L^1(K,\mu)$. By
\[\left|\int_K p(x)~\diff\mu(x)\right| \leq \int_K |p(x)|~\diff\mu(x) \leq \int_K 1~\diff\mu(x) = L(1)<\infty\]
we have $p\in L^1(K,\mu)$ and therefore for any $\varepsilon>0$ there exists a $p_\varepsilon\in\rset[x_1,\dots,x_n]$ such that $p_\varepsilon\leq 1$ on $K$ and
\[\|p-p_\varepsilon\|_{L^1(K,\mu)} = \int_K |p(x) - p_\varepsilon(x)|~\diff\mu(x) \leq \frac{1}{2}\varepsilon.\]
Set $g_\varepsilon := p_\varepsilon^2$. Then
\begin{multline}\label{eq:pg}
L(|g-g_\varepsilon|) = \int_K |g-g_\varepsilon|~\diff\mu(x)
= \int_K |p^2(x) - p_\varepsilon^2(x)|~\diff\mu(x)\\
= \int_K |p-p_\varepsilon|\cdot |p+p_\varepsilon|~\diff\mu(x)
\leq 2\int_K |p(x) - p_\varepsilon(x)|~\diff\mu(x) \leq \varepsilon.
\end{multline}
For $d=0$ we have $g^0 = g_\varepsilon^0 = 1$, i.e., $L(g^0) = L(1) = L(g_\varepsilon^0)$, and for $d=1$ we have $|L(g)-L(g_\varepsilon)|\leq L(|g-g_\varepsilon|)\leq \varepsilon$. So let $d\geq 2$. Then
\begin{align}\label{eq:gg}
|L(g^d) - L(g_\varepsilon^d)| &\leq L(|g^d - g_\varepsilon^d|) = \int_K |g(x)^d - g_\varepsilon(x)^d|~\diff\mu(x)\notag\\
&= \int_K |g(x)-g_\varepsilon(x)|\cdot \left| \sum_{i=0}^{d-1} g(x)^i\cdot g_\varepsilon(x)^{d-1-i} \right|~\diff\mu(x)\\
&\leq d\cdot\int_K |g(x)-g_\varepsilon(x)|~\diff\mu(x) \leq d\cdot\varepsilon.\notag\qedhere
\end{align}
\end{proof}

Note, the $g_\varepsilon$ not only depends on $\varepsilon>0$ but also on $L$ resp.\ its representing measure $\mu$. Since $g$ is measurable (but not necessarily continuous) it is not possible to get $\sup_{x\in K} |g(x) - g_\varepsilon(x)| \leq \varepsilon$. So $g_\varepsilon$ depends on $L$. Otherwise assume we find a $g_\varepsilon\in\rset[x_1,\dots,x_n]$ such that for any moment functional $L$ (with $L(1)=1$), i.e., measure $\mu$ on $K$ with $\mu(K)=1$, we have $\|g-g_\varepsilon\|_{L^1(K,\mu)}\leq \varepsilon$. Then for $\mu =\delta_x$, $x\in K$, we get
\[\sup_{x\in K} |g(x)-g_\varepsilon(x)| = \sup_{x\in K} \|g-g_\varepsilon\|_{L^1(K,\delta_x)} \leq \varepsilon,\]
a contradiction. So the choice of $g_\varepsilon$ depends on $L$ resp.\ $\mu$.

Additionally, note that in fact we can $g_\varepsilon$ not only chose to be a square, but in fact any power: $g_\varepsilon = p_\varepsilon^k$ for a fixed $k\in\nset$. Just replace $p:=\sqrt{g}$ by $p := \sqrt[k]{g}$ in the proof since $g\geq 0$ and use the geometric series as in (\ref{eq:gg}) also in (\ref{eq:pg}).

In \Cref{cor:first} we extended \Cref{thm:first} from a compact and path-connected $K\subset\rset^n$ to an at most countable union of pairwise disjoint, compact, and path-connected $K_i$'s. In \Cref{thm:second} we required that $K$ is a compact and path-connected set. Since we needed compactness of $[0,1]$ in \Cref{thm:second} we can at least extend \Cref{thm:second} to a finite (disjoint) union of compact and path-connected sets.

\begin{cor}\label{cor:second}
Let $k$, $n\in\nset$ be natural numbers and $K\subset\rset^n$ be the union of finitely many compact, path-connected, and pairwise disjoint sets $K_i$: $K = \bigcup_{i=1}^k K_i$. Then there exists a measurable function
\[g:K\to I_k := \bigcup_{i=1}^{k} \left[\frac{2i-2}{2k-1},\frac{2i-1}{2k-1}\right] \subset [0,1]\]
such that for all linear functionals $L: \cV \to \rset$ with $1\in\cV\subseteq C(K,\rset)$ the following are equivalent:
\begin{enumerate}[(i)]
\item $L:\rset[x_1,\dots,x_n]\to\rset$ is a $K$-moment functional.

\item $L:\cV\to\rset$ continuously extends to $\overline{L}:\cV+\rset[g]\to\rset$ such that $\tilde{L}:\rset[t]\to\rset$ defined by $\tilde{L}(t^d) := \overline{L}(g^d)$ for all $d\in\nset_0$ is a $[0,1]$-moment functional.
\end{enumerate}
\end{cor}
\begin{proof}
For all $i=1,\dots,k$ the sets $K_i$ and $[\frac{2i-2}{2k-1},\frac{2i-1}{2k-1}]$ are compact and path-connected and therefore by the \Cref{thm:hahnM} there exist continuous and surjective functions $f_i:[\frac{2i-2}{2k-1},\frac{2i-1}{2k-1}]\to K_i$. By \Cref{lem:souslin} all $K_i$ and $[\frac{2i-2}{2k-1},\frac{2i-1}{2k-1}]$ are Souslin sets and hence by \Cref{thm:jankoff} there exist measurable functions $g_i:K_i\to [\frac{2i-2}{2k-1},\frac{2i-1}{2k-1}]$ such that $f_i(g_i(x))=x$ for all $x\in K_i$, $i=1,\dots,k$. Define
\begin{align*}
f&:I_k\to K=\bigcup_{i=1}^k K_i \quad\text{by}\ f(x) = f_i(x)\ \text{for}\ x\in K_i
\intertext{and}
g&:K=\bigcup_{i=1}^k K_i\to I_k\quad
\text{by}\ g(x) = g_i(x)\ \text{for}\ x\in \left[\frac{2i-2}{2k-1},\frac{2i-1}{2k-1}\right].
\end{align*}
Then $f(g(x)) = x$ for all $x\in K$ and $I_k\subset [0,1]$.

(i)$\to$(ii) and (ii)$\to$(i) are verbatim the same as in the proof of \Cref{thm:second}.
\end{proof}

We are again facing the problem, that $g$ is measurable but not necessarily a polynomial. But as in \Cref{thm:secondApprox} we can approximate $g$ by polynomials.

\begin{cor}\label{cor:secondApprox}
Let $n,k\in\nset$ be natural numbers, $K\subset\rset^n$ the union of finitely many compact, path-connected, and pairwise disjoint sets $K_i$, $K = \bigcup_{i=1}^k K_i$, and let $g:K\to I_k$ be from \Cref{cor:second}. Then for any $\varepsilon>0$ and $K$-moment functional\linebreak \mbox{$L:\rset[x_1,\dots,x_n]\to\rset$} there exists a polynomial $g_\varepsilon\in\rset[x_1,\dots,x_n]$ such that
\[L(|g_\varepsilon-g|) \leq \varepsilon
\qquad\text{and}\qquad
|L(g^d) - L(g_\varepsilon^d)| \leq d\cdot L(|g-g_\varepsilon|) \leq d \cdot \varepsilon\]
hold for all $d\in\nset_0$. $g_\varepsilon$ can be chosen to be a square: $g_\varepsilon = p_\varepsilon^2$, $p_\varepsilon\in\rset[x_1,\dots,x_n]$.
\end{cor}
\begin{proof}
Since $I_k\subset [0,1]$ it is verbatim the same as the proof of \Cref{thm:secondApprox}.
\end{proof}

Note, that in \Cref{thm:secondApprox} and \Cref{cor:secondApprox} we have $|\tilde{L}(t^d)|\leq \tilde{L}(1) = L(1)$, i.e., the error bounds $\leq d\cdot\varepsilon$ exceed $2\cdot\tilde{L}(1)$ at some point and become unreasonable.

We have seen in \Cref{thm:first} resp.\ \Cref{cor:first} that a linear functional $L:\cV\to\rset$ is a $K$-moment functional ($K$ is the countable union of compact and path-connected sets) if and only if it can be transformed by a continuous function $f:I\to K$ to a $I$-moment functional ($I$ is the countable union of intervals $[a_i,b_i]\in\rset$).

If we allow not only continuous functions $f$, then we can generalize this. If we drop continuity of $f$ but add bijectivity almost everywhere we find that any functional on a Borel set of $\rset^n$ is a moment functional if and only if we can transform it into a moment functional with representing measure ``Lebesgue measure on $[0,1]$ plus countably many point evaluations'', see (\ref{eq:measureUnitInt}).

\begin{thm}\label{thm:lebPlusPoints}
Let $n\in\nset$ be a natural number, $B\in\fB(\rset^n)$ be a Borel set, and $\cV$ be a vector space of real measurable functions on $B$ with $1\in\cV$. Then the following are equivalent.
\begin{enumerate}[(i)]
\item $L:\cV\to\rset$ is a $B$-moment functional.

\item There exist Borel sets $M\in\fB(B)$ and $N\in\fB([0,1])$ and a bijective and measurable function (isomorphism) $f:[0,1]\setminus N\to B\setminus M$ such that
\begin{equation}\label{eq:measureUnitInt}
L(v) = \int_0^1 v(f(t))~\diff\nu(t) \qquad\text{with}\qquad \nu = c\cdot\lambda + \sum_{i\in\nset} c_i\cdot\delta_{1/i}
\end{equation}
for all $v\in\cV$, where $c$, $c_i\geq 0$ and $c + \sum_{i\in\nset} c_i = L(1)$, i.e., $\nu\circ f^{-1}$ is a representing measure of $L$.
\end{enumerate}
\end{thm}
\begin{proof}
(ii)$\to$(i): Clear since $\nu\circ f^{-1}$ is a representing measure of $L$.

(i)$\to$(ii): Let $\mu$ be a representing measure of $L$. Then $(B,\fB(B),\mu)$ is by \Cref{exm:lebRoh} a Lebesgue--Rohlin space and therefore by \Cref{thm:lebRoh} isomorph $\modzero$ to $([0,1],\fB([0,1]),\nu)$ with $\nu$ as in (\ref{eq:measureUnitInt}), i.e., there exist Borel sets $M\in\fB(B)$ and $N\in\fB([0,1])$ and a bijective and measurable function $f:[0,1]\setminus N\to B\setminus M$ such that $\nu = \mu\circ f$ and $\mu(M)=\nu(N)=0$. Then by \Cref{lem:integralTrans} for all $v\in\cV$ we have
\begin{align*}
L(v) &= \int_B v(x)~\diff\mu(x) = \int_{B\setminus M} v(f\circ f^{-1})~\diff\mu(x)\\ &= \int_{[0,1]\setminus N} v(f(t))~\diff(\mu\circ f)(t)=\int_0^1 v(f(t))~\diff\nu(t).\qedhere
\end{align*}
\end{proof}

If we drop bijectivity almost everywhere for $f$ then we get \Cref{thm:intro}, i.e., in (\ref{eq:measureUnitInt}) we can chose $c=L(1)$ and $c_i=0$ for all $i\in\nset$.

In \Cref{thm:intro} and \Cref{thm:lebPlusPoints} we can only ensure that $f$ is measurable, but not necessarily continuous or even a polynomial map. The reason is that we can not control the support of a representing measure of $L$. In \Cref{thm:first} we already showed that $f$ can be chosen as continuous and surjective, independent on $L$. But if we restrict the moment functionals resp.\ the support of a representing measure and chose $f$ tailor made for each $K$-moment functional, then $f$ can be chosen to be continuous and surjective and the representing measure will be the Lebesgue measure $\lambda$ on $[0,1]$.

\begin{thm}\label{thm:contFLebesgue}
Let $n\in\nset$, $K\subset\rset^n$ be a compact and path-connected set, $\cV$ be a vector space of real function on $K$, and $L:\cV\to\rset$ be a linear functional. Then the following are equivalent:
\begin{enumerate}[(i)]
\item $L:\cV\to\rset$ is a $K$-moment functional with representing measure $\mu$ such that $\supp\mu = K$.

\item There exists a continuous and surjective function $f:[0,1]\to K$ such that
\[L(v) = \int_0^1 v(f(t))~\diff\lambda(t)\]
for all $v\in\cV$ where $\lambda$ is the Lebesgue measure on $[0,1]$, i.e.,
\[L \quad\overset{f}{\trans}\quad L_{\text{Leb}}:\cL^1([0,1],\lambda)\to\rset.\]
\end{enumerate}
\end{thm}
\begin{proof}
(i)$\to$(ii): Let $L:\cV\to\rset$ be a $K$-moment functional and let $\mu$ be its unique representing measure with $\supp\mu = K$. Since $K$ is a compact and path-connected set, by the \Cref{thm:hahnM} there exists a continuous and surjective function $\tilde{f}:[0,1]\to K$. By \Cref{thm:lebIso} there exists a continuous and surjective function $f:[0,1]\to K$ such that $\mu = \lambda\circ f^{-1}$. For all $v\in\cV$ we get
\[L(p) = \int_K p(x)~\diff\mu(x) = \int_K p(x)~\diff(\lambda\circ f^{-1})(x) \overset{\text{\Cref{lem:integralTrans}}}{=} \int_0^1 p(f(t))~\diff\lambda(t).\tag{$*$}\]

(ii)$\to$(i): By $(*)$ $\mu = \lambda\circ f^{-1}$ is a representing measure of $L$, i.e., $L$ is a $K$-moment functional. To show that $\supp\mu = K$ holds, let $U\subseteq K$ be open. Since $f$ is continuous, $f^{-1}(U)\subseteq [0,1]$ is open and therefore $\mu(U) = \lambda(f^{-1}(U)) > 0$.
\end{proof}

So far we transformed moment functionals to $[0,1]$-moment functionals. We have seen that e.g.\ $\rset^n$-moment functionals can not be continuously transformed into $[0,1]$-moment functionals. But we can transform $\rset^n$-moment functionals continuously into $[0,\infty)$-moment functionals. We need the following.

\begin{lem}\label{lem:RnSpaceFilling}
Let $n\in\nset$ and $\varepsilon>0$. Then there exists a continuous and surjective function $f_\varepsilon:[0,\infty)\to\rset^n$ with
\[t - \varepsilon\leq \|f_\varepsilon (t)\| \leq t + \varepsilon\]
for all $t\geq 0$ and there exists a measurable function $g_\varepsilon :\rset^n\to[0,\infty)$ such that
\[f_\varepsilon(g_\varepsilon(x))=x\qquad\text{and}\qquad\|x\|-\varepsilon\leq g_\varepsilon(x)\leq\|x\|+\varepsilon\]
for all $x\in\rset^n$.
\end{lem}
\begin{proof}
Set
\[A_n := \{x\in\rset^n \,|\, (n-1)\cdot\varepsilon\leq \|x\|\leq n\cdot\varepsilon\}\]
for all $n\in\nset$. Then all $A_n$'s are compact and path-connected and by the \Cref{thm:hahnM} there exist continuous and surjective functions $f_{\varepsilon, n}:[(n-1)\cdot\varepsilon,n\cdot\varepsilon]\to A_n$ for all $n\in\nset$ such that $f_{\varepsilon, n}(n\cdot\varepsilon) = f_{\varepsilon,n+1}(n\cdot\varepsilon)$, i.e., $\|f_{\varepsilon, n}(n\cdot\varepsilon)\|=\|f_{\varepsilon,n+1}(n\cdot\varepsilon)\|=n\cdot\varepsilon$ for all $n\in\nset$. Since $\rset^n = \bigcup_{n\in\nset} A_n$ define $f_\varepsilon:[0,\infty)\to\rset^n$ by $f_\varepsilon|_{[n-1,n]} := f_{\varepsilon,n}$. Then for $t\in [(n-1)\cdot\varepsilon,n\cdot\varepsilon]$ we have
\[t-\varepsilon \leq (n-1)\cdot\varepsilon\leq \|f_\varepsilon(t)\| = \|f_{\varepsilon,n}(t)\|\leq n\cdot\varepsilon \leq t+\varepsilon.\tag{$*$}\]

Since $f:[0,\infty)\to\rset^n$ is surjective and $[0,\infty)$ and $\rset^n$ are Souslin sets by \Cref{lem:souslin} then by \Cref{thm:jankoff} there exists a $g_\varepsilon:\rset^n\to [0,\infty)$ with $f_\varepsilon(g_\varepsilon(x))=x$ for all $x\in\rset^n$. ($*$) implies
\[g_\varepsilon(x)-\varepsilon \leq \|x\| = \|f_\varepsilon(g_\varepsilon(x))\| \leq g_\varepsilon(x) + \varepsilon\]
and therefore $\|x\|-\varepsilon \leq g_\varepsilon(x) \leq \|x\| + \varepsilon$ for all $x\in\rset^n$.
\end{proof}

Similar to \Cref{thm:first} we then get the continuous transformation into $[0,\infty)$-moment functionals.

\begin{thm}\label{thm:last}
Let $n\in\nset$, $f:[0,\infty)\to\rset^n$ be a continuous and surjective function, and $\cV$ be a vector space of measurable functions on $\rset^n$. Then for all linear functionals $L:\cV\to\rset$ the following are equivalent:
\begin{enumerate}[(i)]
\item $L:\cV\to\rset$ is a moment functional.

\item $\tilde{L}:\cV\circ f\to\rset$ defined by $\tilde{L}(v\circ f) := L(v)$ is a $[0,\infty)$-moment functional.
\end{enumerate}
I.e., $L\sctrans\tilde{L}$.
If $\tilde{\mu}$ is a representing measure of $\tilde{L}$, then $\tilde{\mu}\circ f^{-1}$. There exists a function $g:\rset^n\to [0,\infty)$ such that $f(g(x))=x$ for all $x\in\rset^n$ and if $\mu$ is a representing measure of $L$, then $\mu\circ g^{-1}$ is a representing measure of $\tilde{L}$.
\end{thm}
\begin{proof}
Since $\rset^n$ and $[0,\infty)$ are Souslin sets and $f$ is surjective, by \Cref{thm:jankoff} there exists a function $g:\rset^n\to [0,\infty)$ such that $f(g(x))=x$ for all $x\in\rset^n$. It follows that $\tilde{L}$ is well defined by $\tilde{L}(v\circ f) = L(v)$.

(i)$\to$(ii): Let $\mu$ be a representing measure of $L$, then
\begin{multline*}
\tilde{L}(v\circ f) = L(v) = \int_{\rset^n} v(x)~\diff\mu(x) = \int_{\rset^n} v(f(g(x)))~\diff\mu(x)\\ \overset{\text{\Cref{lem:integralTrans}}}{=} \int_0^\infty (v\circ f)(t)~\diff(\mu\circ g^{-1})(t),
\end{multline*}
i.e., $\mu\circ g^{-1}$ is a representing measure of $\tilde{L}$.

(ii)$\to$(i): Let $\tilde{\mu}$ be a representing measure of $\tilde{L}$, then
\[L(v) = \tilde{L}(v\circ f) = \int_0^\infty (v\circ f)(t)~\diff\tilde{\mu}(t) \overset{\text{\Cref{lem:integralTrans}}}{=} \int_{\rset^n} v(x)~\diff(\tilde{\mu}\circ f^{-1})(x),\]
i.e., $\tilde{\mu}\circ f^{-1}$ is a representing measure of $L$.
\end{proof}

\begin{rem}\label{rem:last}
Similar to \Cref{thm:second} we get that for any $\varepsilon>0$ and $g_\varepsilon$ from \Cref{lem:RnSpaceFilling}
\begin{enumerate}[(i)]
\item $L:\rset[x_1,\dots,x_n]\to\rset$ is a moment functional
\end{enumerate}
implies that
\begin{enumerate}[(i)]\setcounter{enumi}{1}
\item $L:\rset[x_1,\dots,x_n]\to\rset$ continuously extends to $\overline{L}:\rset[x_1,\dots,x_n,g]\to\rset$ such that $\tilde{L}:\rset[t]\to\rset$ defined by $\tilde{L}(t^d) := \overline{L}(g^d)$ is a $[0,\infty)$-moment functional, i.e.,
\begin{equation*}
\begin{matrix}
&& L:\rset[x_1,\dots,x_n]\to\rset\\
&& \downtrans\, {\scriptstyle\id_\cX}\\
\tilde{L}:\rset[t]\to\rset & \overset{g}{\trans} & \overline{L}:\rset[x_1,\dots,x_n,g]\to\rset.
\end{matrix}
\end{equation*}
\end{enumerate}
That follows easily from the fact that $0\leq g_\varepsilon(x) \leq \|x\| + \varepsilon \leq \|x\|^2 + 1 + \varepsilon\in\rset[x_1,\dots,x_n]$. However, it is open whether the strong direction (ii)$\to$(i) as in \Cref{thm:second} holds in general. In \Cref{thm:second} compactness of $K$ implied that $\rset[x_1,\dots,x_n]$ is dense in $C(K,\rset)$ and hence $f$ could be approximated and the representing measure of $L$ is unique. On $\rset^n$ both do not hold and hence (ii)$\to$(i) can so far not be ensured in the same fashion as in \Cref{thm:second}.\exmsym
\end{rem}

At the end of this section we want to discuss two things that can easily be missed. The first is a crucial technical remark and the second is a historical one.

For most transformations $\trans$ we required that $f:\cX\to\cY$ is surjective to apply \Cref{thm:jankoff} to get a right-side inverse $g:\cY\to\cX$, i.e., $f(g(y))=y$ for all $y\in\cY$. E.g.\ in \Cref{thm:second} we used this $g$ directly to embed a $[0,1]$-moment functional into an extension $\overline{L}$ of $L$. However, for any $f:\cX\to\cY$ of course $f:\cX\to f(\cY)$ is surjective. If $f$ is continuous and $\cX$ Borel, then $f(\cX)$ remains even a Borel set. Otherwise $f(\cX)$ is at least a Souslin set.

To demonstrate, that $f:\cX\to\cY$ needs to be surjective and the restriction $f:\cX\to f(\cX)$ can not be used, let $L:\rset[x_1,\dots,x_n]\to\rset$ be a linear functional such that $L(p^2)\geq 0$ for all $p\in\rset[x_1,\dots,x_n]$. Let $f\in\rset[x_1,\dots,x_n]$, then define $\tilde{L}:\rset[t]\to\rset$ by $\tilde{L}(t^d) := L(f^d)$ for all $d\in\nset_0$. We have $\tilde{L}(p^2) = L((p\circ f)^2) \geq 0$ for all $p\in\rset[t]$, i.e., $\tilde{L}$ is a Hamburger moment functional and there exists a measure $\nu$ on $\rset$ such that
\[\tilde{L}(p) = \int_\rset p(t)~\diff\nu(t) \qquad\text{for all}\ p\in\rset[t],\]
i.e.,
\begin{equation}\label{eq:extensionWrong}
L(f^d) = \tilde{L}(t^d) = \int_\rset t^d~\diff\nu(t) \qquad\text{for all}\ d\in\nset_0.
\end{equation}
The important thing is, that (\ref{eq:extensionWrong}) does not imply that there exists a $\mu$ such that $L(f^d) = \int_{\rset^n} f^d(x)~\diff\mu(x)$ for all $d\in\nset_0$. \Cref{thm:jankoff} incorrectly applied in (X) would suggest that there is a $g$ such that $f(g(t))=t$, i.e.,
\[\int_\rset t^d~\diff\nu(t) \overset{\text{(X)}}{=} \int_\rset f(g(t))^d~\diff\nu(t) = \int_{\rset^n} f(x)^d~\diff(\nu\circ g^{-1})(x)\]
and hence $\nu\circ g^{-1}$ is a representing measure for $L(f^d)$. Therefore (X) would imply $L(f)\geq 0$ for all $f\in\rset[x_1,\dots,x_n]$ with $f\geq 0$ since $\nu\circ g^{-1}$ is non-negative. Havilands Theorem then shows that $L$ is a moment functional. But for $L$ we only had $L(p^2)\geq 0$ for all $p\in\rset[x_1,\dots,x_n]$ and for $n\geq 2$ there are functionals only with $L(p^2)\geq 0$ which are not moment functionals \cite{berg79,schmud79,friedr84}. This is the contradiction. We have to ensure, that $\supp\nu\subseteq f(\rset^n)$ holds to apply \Cref{thm:jankoff}.

For the historical remark, in this study we frequently encountered the case where a linear functional $L:\cV\to\rset$ (or its transformation) lives on measurable functions $\cV$, i.e., we apparently face the problem that our functions $v\in\cV$ live on a measuable space $(\cX,\cA)$. But a main tool in the moment problem is the Riesz Representation Theorem and it works with (compactly supported) continuous functions on locally compact Hausdorff spaces. While the linear functional is extended to compactly supported continuous functions via e.g.\ the Hahn--Banach Theorem, changing or extending a measurable space $(\cX,\cA)$ to a topological space, especially to a locally compact Hausdorff space, is in general not possible. Another important case where we rather work on a measurable space than a locally compact Hausdorff space is the Richter Theorem.

\begin{thmRichter}[see {\cite[Satz 4]{richte57}}]
Let $\cV$ be a finite-dimensional vector space of measurable functions on a measurable space $(\cX,\cA)$. Then every moment functional $L:\cV\to\rset$ has a finitely atomic representing measure
\[\sum_{i=1}^d c_i\cdot \delta_{x_i}\]
with $c_i>0$, $x_i\in\cX$, and $d = \dim\cV$.
\end{thmRichter}

This theorem of Richter from 1957 was in the broader mathematical community not known (despite the fact that it is stated in this generality e.g.\ in \cite[Thm.\ 1]{kemper68} and more recently in \cite[pp.\ 198--199]{floudas01}). Several attempts where made to generalize a much weaker result from  \cite{tchaka57}. Richter's Theorem can also be called \emph{Richter--Rogosinski--Rosenbloom Theorem} to account for all contributions \cite{richte57,rogosi58,rosenb52}. See \cite{didioConeArXiv} for more on the early history of this theorem.

We include \Cref{sec:daniell} to avoid a similar confusion how to handle the representations of linear functionals of measurable functions which live not necessarily on a locally compact Hausdorff space. This question was already fully answered by P.\ J.\ Daniell in 1918 \cite{daniell18}. To our knowledge this result does not appear in any standard functional analytic textbooks or works on the moment problem. It is treated e.g.\ in \cite[Ch.\ 7.8]{bogachevMeasureTheory} and \cite[Ch.\ 2.5]{federerGeomMeasTheo}. Especially the approach in \cite{federerGeomMeasTheo} via the outer measure gives a simple proof of the general statement which works without any completions in the lattice of functions.

\begin{dfn}\label{dfn:lattice}
Let $\cX$ be a space. We call a set $\cF$ of functions $f:\cX\to\rset$ a \emph{lattice} (\emph{of functions}) if the following holds:
\begin{enumerate}[i)]
\item $c\cdot f\in\cF$ for all $c\geq 0$ and $f\in\cF$,
\item $f+g\in\cF$ for all $f,g\in\cF$,
\item $\inf(f,g)\in\cF$ for all $f,g\in\cF$,
\item $\inf(f,c)\in\cF$ for all $c\geq 0$ and $f\in\cF$, and
\item $g-f\in\cF$ for all $f,g\in\cF$ with $f\leq g$.
\end{enumerate}
\end{dfn}

Some authors require that a lattice of functions is a vector space. But for proving \Cref{thm:daniell} it is only necessary that a lattice is a cone.

\begin{thmDaniell}[P.\ J.\ Daniell 1918 \cite{daniell18}]\label{thm:daniell}
Let $\cF$ be a lattice of functions on a space $\cX$ and let $L:\cF\to\rset$ be such that
\begin{enumerate}[i)]
\item $L(f+g) = L(f) + L(g)$ for all $f,g\in\cF$,
\item $L(c\cdot f) = c\cdot L(f)$ for all $c\geq 0$ and $f\in\cF$,
\item $L(f) \leq L(g)$ for all $f,g\in\cF$ with $f\leq g$,
\item $L(f_n)\nearrow L(g)$ as $n\to\infty$ for all $g\in\cF$ and $f_n\in\cF$ with $f_n\nearrow g$.
\end{enumerate}
There exists a measure $\mu$ on $(\cX,\cA)$ with $\cA := \sigma(\{f^{-1}((-\infty,a]) \,|\, a\in\rset,\ f\in\cF\})$ such that
\[L(f) = \int_\cX f(x)~\diff\mu(x)\]
for all $f\in\cF$.
\end{thmDaniell}

The most impressive part is that the functional $L:\cF\to\rset$ lives only on a lattice $\cF$ of functions $f:\cX\to\rset$ where $\cX$ is a set without any structure. \Cref{thm:daniell} provides a representing measure $\mu$ including the\linebreak \mbox{$\sigma$-algebra} $\cA$ for the measurable space $(\cX,\cA)$.

The proof of \Cref{thm:daniell} we give in \Cref{sec:daniell} is taken from \cite[Thm.\ 2.5.2]{federerGeomMeasTheo} with alterations to fit in the non-outer-measure approach.

Riesz Representation Theorem follows directly from \Cref{thm:daniell}. $C_0(\cX,\rset)$, $\cX$ a locally compact Hausdorff space, is a lattice of functions, (i) and (ii) are the linearity of $L$, (iii) non-negativity of $L$, and the continuity condition (iv) of $L$ follows easily from uniform convergence in $C_0(\cX,\rset)$.

\section{Conclusion and Open Questions}
\label{sec:open}

We end with some conclusions, outlook, and some open questions which appeared during our investigation.

We gained in \Cref{sec:trans} basic properties of the transformation $\trans$ of linear functioals. Especially in \Cref{thm:transMomProp} that a strong transformation $L\strans K$ implies that $L$ is a moment functional if and only if $K$ is a moment functional. In \Cref{lem:backTrans} we have seen that $L\strans K$ implies the weaker statements $L\trans K$ and $K\trans L$. So it is natural to ask if the reverse holds.

\begin{open}
Does $L\trans K$ and $K\trans L$ imply $L\strans K$?
\end{open}

Additionally, can the requirement of a strong transformation be weakened? While we have seen that surjectivity of $f:\cX\to\cY$ is necessary and can in general not be omitted, it should be possible to weaken the condition that $\cV\circ f = \cU$ from $L:\cV\to\rset$ and $K:\cU\to\rset$. It is in fact only necessary that $\cV$ and $\cU$ (and therefore $L$ and $K$) can be extended to some $\overline{\cV}\supseteq\cV$ and $\overline{\cU}\supseteq\cU$ such that $\overline{\cV}\circ f = \overline{\cU}$.

In \Cref{prop:oneLebesgueDirection} we have seen that for a moment functional $L$ with an atomless representing measure there exists an integrable function $f$ such that $L$ extended to $\overline{L}:\cV+\rset[f]\to\rset$ which obeys $\overline{L}|_{\rset[f]} = L_{\text{Leb}}$, i.e., $\overline{L}(f^d) = \frac{L(1)}{d+1}$ for all $d\in\nset_0$. Because of the simplicity of $L_{\text{Leb}}$ in \Cref{exm:lebOnUnit}, are there other ``directions'', i.e., $f$'s, with similar properties?

\begin{open}
Are there other ``directions'' $f$ with $\overline{L}(f^d) = \frac{L(1)}{d+1}$ or a similar behavior?
\end{open}

The importance of this question is again revealed in \Cref{thm:second} where we have a similar structure in (\ref{eq:directionTrans}):
\begin{equation*}
\begin{matrix}
&& L:\cV\to\rset\\
&& \downtrans\, {\scriptstyle\id_\cX}\\
\tilde{L}:\rset[t]\to\rset & \overset{g}{\trans} & \overline{L}:\cV+\rset[g]\to\rset.
\end{matrix}
\end{equation*}
There exists a function $g:K\to[0,1]$ such that: A linear functional $L:\cV\to\rset$ is a $K$-moment problem if and only if it continuously extends to some $\overline{L}:\cV+\rset[g]\to\rset$ and $\tilde{L}:\rset[t]\to\rset$ defined by $\tilde{L}(t^d):=\overline{L}(g^d)$ for all $d\in\nset_0$ is a $[0,1]$-moment functional.

At this point the reader shall be reminded of the following functional analytic fact. Let $L:\rset[x_1,\dots,x_n]\to\rset$ be a linear functional with $L(p^2)\geq 0$ for all $p\in\rset[x_1,\dots,x_n]$. $(\cset[x_1,\dots,x_n],\langle\,\cdot\,,\,\cdot\,\rangle)$ with $\langle p,q\rangle:= L(p\cdot \overline{q})$ is a pre-Hilbert space via complexification of $L$ by linearity (and removing the possible kernel of $L$), and for all $i=1,\dots,n$ the multiplication operators $X_i$ are defined by $(X_ip)(x_1,\dots,x_n) := x_i\cdot p(x_1,\dots,x_n)$ for all $p\in\cset[x_1,\dots,x_n]$. $(X_1,\dots,X_n)$ is a tuple of commuting symmetric operators on $(\cset[x_1,\dots,x_n],\langle\,\cdot\,,\,\cdot\,\rangle)$. Then $L$ is a moment functional if and only if $(X_1,\dots,X_n)$ extends to a tuple $(\overline{X_1},\dots,\overline{X_n})$ of communting self-adjoint operators on some Hilbert space $\cH \supset (\cset[x_1,\dots,x_n],\langle\,\cdot\,,\,\cdot\,\rangle)$.

But extending $L$ to $\rset[x_1,\dots,x_n,g]\supseteq\rset[x_1,\dots,x_n]+\rset[g]$ gives
\begin{equation*}
\begin{matrix}
&& L:\rset[x_1,\dots,x_n]\to\rset\\
&& \downtrans\, {\scriptstyle\id_\cX}\\
\tilde{L}:\rset[t]\to\rset & \overset{g}{\trans} & \overline{L}:\rset[x_1,\dots,x_n,g]\to\rset.
\end{matrix}
\end{equation*}
By \Cref{thm:second} it is sufficient to ensure that the multiplication operator $G$ on $\cset[x_1,\dots,x_n,g]$, i.e., $(Gp)(x) := g(x)\cdot p(x)$, has a self-adjoint extension. So the tuple $(X_1,\dots,X_n)$ is replaced by $G$ and the open question is loosely the following:

\begin{open}
What is the functional analysis behind the $g$ in \Cref{thm:second}?
\end{open}

Note, that in the setting of \Cref{thm:second} the multiplication operators are bounded since $K$ is compact. In the setup of $K=\rset^n$, see \Cref{rem:last}, we have in general unbounded operators and only the easy direction (i)$\to$(ii) was shown. It is open if (ii)$\to$(i) also holds in the unbounded case.

\begin{open}
Does (ii)$\to$(i) in \Cref{rem:last} holds in general or is there a counter example?
\end{open}

In \Cref{thm:secondApprox} we have seen that this $g$ in \Cref{thm:second} can be approximated by polynomials $g_\varepsilon\in\rset[x_1,\dots,x_n]$. So a natural question (especially in applications) is to ask the following:

\begin{open}
How does $\deg g_\varepsilon$ of $g_\varepsilon$ in \Cref{thm:secondApprox} grow with $\varepsilon\to 0$?
\end{open}

The reason that $g$ in \Cref{thm:second} is only a measurable function but not a polynomial even for $\cV = \rset[x_1,\dots,x_n]$ is a consequence of the reduction of the dimension. We reduce the dimension of $K$, in general $\dim K\geq 2$, to $1$, i.e., the dimension of $[0,1]$. However, a transformation $\overset{f}{\trans}$ not necessarily needs to reduce the dimension of $K$.

To remain in the algebraic setup we have to investigate transformations $\overset{f}{\trans}$ of linear functionals on $\rset[x_1,\dots,x_n]$ where $f$ is a (bi)rational or polynomial function. Since a linear functional $L$ is a moment functional if and only if $L(f)\geq 0$ for all $f\geq 0$ on $K$, $f\in\rset[x_1,\dots,x_n]$, i.e., it has long been known that moment functionals are closely related to non-negative polynomials (Haviland Theorem), these transformations of moment functionals with (bi)rational or polynomial functions might give deeper insight into non-negative polynomials.

\begin{open}
Do transformations $\overset{f}{\trans}$ of moment functionals with polynomial or (bi)rational  $f$ give deeper insight into/characterizations of non-negative polynomials?
\end{open}

\appendix
\section{Daniell's Representation Theorem}\label{sec:daniell}

In this section we give a proof of \Cref{thm:daniell} from 1918 \cite{daniell18} in more recent mathematical notations following the proof in \cite[Thm.\ 2.5.2]{federerGeomMeasTheo}.

\begin{dfn}\label{dfn:outerMeasure}
Let $\cX$ be a set. A set function
$\mu: \cP(\cX)\to [0,\infty]$
with
\begin{enumerate}[i)]
\item $\mu(\emptyset) = 0$,
\item $\mu(A) \leq \mu(B)$ for all $A\subseteq B\subseteq\cX$,
\item $\mu\left(\bigcup_{i=1}^\infty A_i\right) \leq \sum_{i=1}^\infty \mu(A_i)$ for all $A_i\in\cX$
\end{enumerate}
is called a (\emph{Carath\'eodory)} \emph{outer measure}.
\end{dfn}

\begin{dfn}\label{dfn:CaraMeasurable}
For an outer measure $\mu$ on $\cX$ a set $A\subseteq\cX$ is called (\emph{Carath\'eodory}) \emph{$\mu$-measurable} if for every $E\subseteq\cX$ we have
$\mu(E) = \mu(E\cap A) + \mu(E\setminus A)$.
\end{dfn}

\begin{rem}\label{rem:CaraMeasurable}
Since by \Cref{dfn:outerMeasure}(iii) we always have
\[\mu(E) = \mu((E\cap A)\cup (E\setminus A)) \leq \mu(E\cap A) + \mu(E\setminus A)\]
it is sufficient for $\mu$-measurability to test
\begin{equation}\label{eq:measurableTest}
\mu(E) \geq \mu(E\cap A) + \mu(E\setminus A).
\end{equation}
\end{rem}

An outer measure is in fact a measure on all its measurable sets.

\begin{thm}\label{thm:outerMeasure}
Let $\mu$ be an outer measure on a set $\cX$ and $\cA_\mu\subseteq\cP(\cX)$ be the set of all $\mu$-measurable sets. Then $\cA_\mu$ is a $\sigma$-algebra of $\cX$ and $\mu$ is a measure on $(\cX,\cA_\mu)$.
\end{thm}
\begin{proof}
See e.g.\ \cite[Thm.\ 1.11.4(iii)]{bogachevMeasureTheory}.
\end{proof}

Outer measures give another characterization of measurable functions.

\begin{lem}\label{lem:equivMeasurable}
Let $\mu$ be an outer measure on $\cX$ and $f:\cX\to [-\infty,\infty]$ be a function. Then $f$ is $\mu$-measurable if and only if
\[\mu(A) \geq \mu(\{x\in\cX \,|\, f(x)\leq a\}) + \mu(\{x\in\cX \,|\, f(x)\geq b\})\]
for all $A\subseteq\cX$ and $-\infty< a < b < \infty$.
\end{lem}
\begin{proof}
See e.g.\ \cite[2.3.2(7) pp.\ 74/75]{federerGeomMeasTheo}.
\end{proof}

\begin{dfn}\label{dfn:regularmeasure}
An outer measure $\mu$ is called \emph{regular} if for each set $A\subseteq\cX$ there exists a $\mu$-measurable set $B\subseteq\cX$ with $A\subseteq B$ and $\mu(A) = \mu(B)$.
\end{dfn}

Let $f,g:(\cX,\cA)\to\rset$ be two functions. Then we define $\inf(f,g)$ by $\inf(f,g)(x) := \inf(f(x),g(x))$ for all $x\in\cX$ and similarly $\sup(f,g)$. Additionally, $f\leq g$ iff $f(x)\leq g(x)$ for all $x\in\cX$.

%


Given a lattice $\cF$ the following result \cite[2.5.1, p.\ 91]{federerGeomMeasTheo} shows that it induces another lattice $\cF^+$ by taking only the non-negative functions.

\begin{lem}
Let $\cF$ be a non-empty lattice on a space $\cX$ and set
\[\cF^+ := \cF\cap\{f:\cX\to\rset \,|\, f\geq 0\}.\]
Then
\begin{enumerate}[i)]
\item $f^+, f^-,|f|\in\cF^+$ for all $f\in\cF$ and

\item $\cF^+$ is a non-empty lattice on $\cX$.
\end{enumerate}
\end{lem}
\begin{proof}
i): Since $\inf(f,0)\in\cF$ and $\inf(f,0)\leq f$ we have $f^+ = \sup(f,0) = f - \inf(f,0)\in\cF^+$ for all $f\in\cF$. Since $f \leq f^+ = \sup(f,0)\in\cF$ we have $f^- = f^+ - f\in\cF^+$ for all $f\in\cF$. It follows that $|f| = f^+ + f^- \in\cF^+$ for all $f\in\cF$.

ii): Since $\cF$ is non-empty there is a $f\in\cF$ and by (ii) we have $|f|\in\cF$ and hence $|f|\in\cF^+$. $\cF^+$ is a lattice by directly checking the \Cref{dfn:lattice}.
\end{proof}

Note, that $h_n\nearrow g$ means a sequence $(h_n)_{n\in\nset}$ with $h_1 \leq h_2 \leq ... \leq g$, i.e., point-wise non-decreasing, with $\lim_{n\to\infty} h_n(x) = g(x)$ for all $x\in\cX$. Equivalently, $h_n\searrow 0$ denotes a point-wise non-increasing sequence with $\lim_{n\to\infty} h_n(x) = 0$ for all $x\in\cX$.

%
%
%
%
%
\begin{proof}[Proof of \Cref{thm:daniell}]
By assumption (iii) we have $L(f) \geq L(0\cdot f) = 0$ for all $f\in\cF^+$.

For any $A\subseteq\cX$ we say a sequence $(f_n)_{n\in\nset}$ \emph{suits} $A$ if and only if $f_n\in\cF^+$ and $f_n \leq f_{n+1}$ for all $n\in\nset$ and
\[\lim_{n\to\infty} f_n(x) \geq 1 \qquad\text{for all}\ x\in A.\]
Note, that we can even assume equality by replacing the $f_n$'s by $\tilde{f}_n = \inf(f_n,1)\in\cF^+$. Then we define
\begin{equation}\label{eq:measureDaniellDfn}
\mu(A) := \inf \left\{\lim_{n\to\infty} L(f_n) \;\middle|\; (f_n)_{n\in\nset}\ \text{suits}\ A\right\}
\end{equation}
which is $\infty$ if there is no sequence $(f_n)_{n\in\nset}$ that suits $A$.

We prove that $\mu$ is an outer measure, see \Cref{dfn:outerMeasure}. By assumption (iii) $L(f_n)$ is a non-negative increasing sequence and therefore $\lim_{n\to\infty} L(f_n)$ exists and is in $[0,\infty]$ and therefore $\mu:\cP(\cX)\to [0,\infty]$. For $A=\emptyset$ the zero sequence $f_n=0\in\cF^+$ is suited and therefore $\mu(\emptyset) = 0$. Let $A\subseteq B\subseteq\cX$, then a suited sequence $(f_n)_{n\in\nset}$ of $B$ is also a suited sequence for $A$ and therefore $\mu(A)\leq\mu(B)$. Let $A_i\subseteq\cX$, $i\in\nset$, and set $A := \bigcup_{i=1}^\infty A_i$. Any suited sequence for $A$ is a suited sequences for all $A_i$. Assume there is an $A_i$ which has no suited sequence, then $A$ has no suited sequence and $\mu(A) = \infty \leq \sum_{i=1}^\infty \mu(A_i) = \infty$. So assume all $A_i$ have suited sequences, say $(f_{i,n})_{n\in\nset}$ suits $A_i$, $i\in\nset$. Then $f_n := \sum_{i=1}^n f_{i,n}$ suits $A$ and
\[\mu(A) \leq \lim_{n\to\infty} L(f_n) = \lim_{n\to\infty} \sum_{i=1}^n L(f_{i,n}) \leq \sum_{i=1}^\infty \lim_{m\to\infty} L(f_{i,m}).\]
Taking the infimum on the right side for all $A_i$'s retains the inequality and gives
\[\mu\left(\bigcup_{i=1}^\infty A_i\right) = \mu(A) \leq \sum_{i=1}^\infty \mu(A_i).\]
Hence, all conditions in \Cref{dfn:outerMeasure} are fulfilled and $\mu$ is an outer measure.

Since $\mu$ is an outer measure on $\cX$ by \Cref{thm:outerMeasure} the set $\tilde{\cA}$ of all $\mu$-measurable sets of $\cX$ is a $\sigma$-algebra and $\mu$ is a measure on $(\cX,\tilde{\cA})$.

It remains to show that all $f\in\cF$ are $\mu$-measurable, $\mu$ is a measure on $(\cX,\cA)$ with $\cA = \sigma(\{f^{-1}((-\infty,a]) \,|\, a\in\rset,\ f\in\cF\})$, and $L(f) = \int_\cX f(x)~\diff\mu(x)$ for all $f\in\cF$.

Since $f = f^+ - f^-$ with $f^+,f^-\in\cF^+$ it is sufficient to show that every function in $\cF^+$ is $\mu$-measurable. So let $f\in\cF^+$. To show that $f$ is $\mu$-measurable it is sufficient to show that $A := f^{-1}((-\infty,a]) = \{x\in\cX \,|\, f(x)\leq a\}\in\cA$ for all $a\in\rset$, i.e., $A$ is $\mu$-measurable by \Cref{dfn:CaraMeasurable} resp.\ \Cref{rem:CaraMeasurable} if (\ref{eq:measurableTest}) holds for all $E\subseteq\cX$. From $E\setminus A = E\cap (\cX\setminus A) = E\cap\{x\in\cX \,|\, f(x)> a\}$ we have to verify
\[\mu(E) \geq \mu(E\cap\{x\in\cX \,|\, f(x)\leq a\})  + \mu(E\cap\{x\in\cX \,|\, f(x)> a\})\]
and by \Cref{lem:equivMeasurable} this is equivalent to
\[\mu(E) \geq \mu(\underbrace{E\cap\{x\in\cX \,|\, f(x)\leq a\}}_{=:E_a})  + \mu(\underbrace{E\cap\{x\in\cX \,|\, f(x)\geq b\}}_{=: E_b}) \tag{$*$}\]
for all $a<b$. For $a<0$ or $\mu(E)=\infty$ ($*$) is trivial, so assume $a\geq 0$ and $\mu(E)<\infty$.

Let $(g_n)_{n\in\nset}$ be a sequence that suits $E$ and set
\[h := (b-a)^{-1}\cdot [\inf(f,b)-\inf(f,a)]\in\cF^+ \qquad\text{and}\qquad k_n := \inf(g_n,h)\in\cF^+.\]
Then we have $0\leq k_{n+1} - k_n \leq g_{n+1} - g_n$,
\begin{align*}
h(x)&=1  \qquad\text{for all}\ x\in\cX\ \text{with}\ f(x)\geq b,
\intertext{and}
h(x)&=0 \qquad\text{for all}\ x\in\cX\ \text{with}\ f(x)\leq a.
\end{align*}
It follows that $(k_n)_{n\in\nset}$ suits $E_b$ and $(g_n-k_n)_{n\in\nset}$ suits $E_a$. Therefore,
\[\lim_{n\to\infty} L(g_n) = \lim_{n\to\infty} [L(g_n-k_n) + L(k_n)] \geq \mu(E_a) + \mu(E_b)\]
and taking the infimum on the left side retains the inequality and proves ($*$). Hence, all $f\in\cF^+$ and therefore all $f\in\cF$ are $\mu$-measurable.

Let us show that $\mu$ remains a measure on $(\cX,\cA)$. Since all $f\in\cF$ are $\mu$- and $\cA$-measurable we have
\[f^{-1}((-\infty,a])\in\tilde{\cA}\]
for all $a\in\rset$ and $f\in\cF$. Therefore,
\[\cA_\mu := \sigma(\{f^{-1}((-\infty,a]) \,|\, a\in\rset,\ f\in\cF\}) \subseteq\tilde{\cA}\]
is a $\sigma$-algebra and we can restrict $\mu$ resp.\ $\tilde{\cA}$ to $\cA$. $\mu$ is a measure on $(\cX,\cA)$.

We show that $L(f) = \int_\cX f(x)~\diff\mu(x)$ holds for all $f\in\cF^+$. Let $f\in\cF^+$ and set
\[f_t := \inf(f,t)\]
for $t\geq 0$. If $\varepsilon>0$ and $k\in\nset$ then
\begin{align*}
0 \leq f_{k\varepsilon}(x) - f_{(k-1)\varepsilon}(x) &\leq \varepsilon \quad\text{for all}\ x\in\cX,\\
f_{k\varepsilon}(x) - f_{(k-1)\varepsilon}(x) &= \varepsilon \quad\text{for all}\ x\in\cX\ \text{with}\ f(x)\geq k\varepsilon,
\intertext{and}
f_{k\varepsilon}(x) - f_{(k-1)\varepsilon}(x) &= 0 \quad\text{for all}\ x\in\cX\ \text{with}\ f(x)\leq (k-1)\varepsilon.
\end{align*}
The constant sequence $(\varepsilon^{-1}\cdot(f_{k\varepsilon} - f_{(k-1)\varepsilon}))_{n\in\nset}$ suits $\{x\in\cX \,|\, f(x)\geq k\varepsilon\}$ and consequently
\begin{alignat*}{2}
L(f_{k\varepsilon} - f_{(k-1)\varepsilon}) &\geq \varepsilon\cdot\mu(\{x\in\cX \,|\, f(x)\geq k\varepsilon\})\\
&\geq \int_\cX f_{(k+1)\varepsilon}(x) - f_{k\varepsilon}(x)~\diff\mu(x)\\
&\geq \varepsilon\cdot \mu(\{x\in\cX \,|\, f(x)\geq (k+1)\varepsilon\})
&&\geq L(f_{(k+2)\varepsilon} - f_{(k+1)\varepsilon}).
\intertext{Summing with respect to $k$ from $1$ to $n$ we find}
L(f_{n\varepsilon}) &\geq \;\;\int_\cX f_{(n+1)\varepsilon}(x) - f_\varepsilon(x)~\diff\mu(x) &&\geq L(f_{(n+2)\varepsilon} - f_{2\varepsilon})
\intertext{and since $f_{n\varepsilon}\nearrow f$ as $n\to\infty$ we get from assumption (iv) for $n\to\infty$}
L(f) &\geq \qquad\int_\cX f(x) - f_\varepsilon(x)~\diff\mu(x) &&\geq L(f - f_{2\varepsilon})
\intertext{which gives again from assumption (iv) for $\varepsilon\searrow 0$}
L(f) & \geq \qquad\quad\;\int_\cX f(x)~\diff\mu(x) &&\geq L(f).
\end{alignat*}
Hence, $L(f) = \int_\cX f(x)~\diff\mu(x)$ for all $f\in\cF^+$.

Finally, for all $f\in\cF$ we have $f = f^+ - f^-$ with $f^+,f^-\in\cF^+$ which implies
\[\int_\cX f(x)~\diff\mu(x) = \int_\cX f^+(x)~\diff\mu(x) - \int_\cX f^-(x)~\diff\mu(x) = L(f^+) - L(f^-) = L(f)\]
where the last equality follows from $f^+ = f + f^-$ and assumption (i).
\end{proof}

\begin{rem}\label{rem:DaniellCondition}
Note that in \Cref{thm:daniell} the assumption (iv) is equivalent to
\begin{enumerate}[\itshape i')]\setcounter{enumi}{3}
\item \textit{$L(h_n)\searrow 0$ as $n\to\infty$ for all $h_n\in\cF$ with $h_n\searrow 0$ as $n\to\infty$}
\end{enumerate}
since $f_n\nearrow g$ implies $f_n\leq g$ and $0\leq h_n=g-f_n\in\cF$:
\[L(g) = L(g-f_n+f_n) = L(g-f_n) + \underbrace{L(f_n)}_{\nearrow L(g)} = \underbrace{L(h_n)}_{\searrow 0} + L(f_n).\]
\end{rem}

The representing measure $\mu$ in \Cref{thm:daniell} is not unique. But the representing measure $\mu$ constructed in (\ref{eq:measureDaniellDfn}) has further properties, see e.g.\ \cite[2.5.3]{federerGeomMeasTheo}. In \cite{daniell18} also a signed version of \Cref{thm:daniell} is proven, see also \cite[Thm.\ 2.5.5]{federerGeomMeasTheo}.


\providecommand{\bysame}{\leavevmode\hbox to3em{\hrulefill}\thinspace}
\providecommand{\MR}{\relax\ifhmode\unskip\space\fi MR }
\providecommand{\MRhref}[2]{%
  \href{http://www.ams.org/mathscinet-getitem?mr=#1}{#2}
}
\providecommand{\href}[2]{#2}

\end{document}